\documentclass{amsart}
\usepackage[all,hyperref,numberbysection]{bi-discrete}
\usepackage{amsmath}
\usepackage{upgreek}
\usepackage{microtype}
\numberwithin{equation}{section}
\usepackage[backend=biber,maxbibnames=5,maxalphanames=5,style=alphabetic,bibencoding=utf8,giveninits,url=false,isbn=false]{biblatex}
\AtBeginBibliography{\small}
\allowdisplaybreaks

\usepackage{a4wide}

\makeatletter
\def\moverlay{\mathpalette\mov@rlay}
\def\mov@rlay#1#2{\leavevmode\vtop{%
   \baselineskip\z@skip \lineskiplimit-\maxdimen
   \ialign{\hfil$\m@th#1##$\hfil\cr#2\crcr}}}
\newcommand{\charfusion}[3][\mathord]{
    #1{\ifx#1\mathop\vphantom{#2}\fi
        \mathpalette\mov@rlay{#2\cr#3}
      }
    \ifx#1\mathop\expandafter\displaylimits\fi}

\makeatother

\DeclareMathOperator*{\osc}{osc}

\begin{document}
\author[Kim]{Kyeongbae Kim}
\address{Department of Mathematical Sciences, Seoul National University, Seoul 08826, Korea}
\email{kkba6611@snu.ac.kr}
\author[Lee]{Ho-Sik Lee}
\address{Fakult\"at f\"ur Mathematik, Universit\"at Bielefeld, Postfach 100131, D-33501 Bielefeld, Germany}
\email{ho-sik.lee@uni-bielefeld.de}
\author[Prasad]{Harsh Prasad}
\address{Fakult\"at f\"ur Mathematik, Universit\"at Bielefeld, Postfach 100131, D-33501 Bielefeld, Germany}
\email{hprasad@math.uni-bielefeld.de}

\makeatletter
\@namedef{subjclassname@2020}{\textup{2020} Mathematics Subject Classification}
\makeatother


\thanks{Kyeongbae Kim is supported from the National Research Foundation of Korea (NRF) through IRTG 2235/NRF-2016K2A9A2A13003815 at Seoul National University. Ho-Sik Lee is supported from funding of the Deutsche Forschungsgemeinschaft (DFG, German Research Foundation) through GRK 2235/2 2021 - 282638148. Harsh Prasad is supported from Deutsche Forschungsgemeinschaft (DFG, German Research Foundation) through Project-ID 317210226 – SFB 1283}

\title[Nonlocal two-phase Stefan problem]{Logarithmic continuity for the Nonlocal degenerate two-phase Stefan problem}

\begin{abstract}
We establish certain oscillation estimates for weak solutions to nonlinear, anomalous phase transitions modeled on the nonlocal two-phase Stefan problem. The problem is singular in time, is scaling deficient and influenced by far-off effects. We study the the problem in a geometry adapted to the solution and obtain oscillation estimates in intrinsically scaled cylinders. Furthermore, via certain uniform estimates, we construct a continuous weak solution to the corresponding initial boundary value problem with a quantitative modulus of continuity.
\end{abstract}

\maketitle


\section{Introduction}
We study the following nonlocal two-phase Stefan problem
\begin{align}\label{eq.main}
    \partial_t(u+\beta(u))+\mathcal{L}u\ni 0\quad\text{in }\Omega\times(0,T]\subset\setR^{n+1},
\end{align}
where $\Omega$ is an open and bounded domain in $\setR^n$ with $n\geq 2$, $u$ is an unknown, and a maximal monotone graph $\beta:\setR\rightarrow[0,1]$ is defined as
\begin{equation*}
\begin{aligned}
    \beta(\xi)=\begin{cases}
        1\quad&\text{if }\xi>0,\\
        [0,1]\quad&\text{if }\xi=0,\\
        0\quad&\text{if }\xi<0.
    \end{cases}
\end{aligned}
\end{equation*}
In addition, the nonlocal operator $\mathcal{L}$ is defined by
\begin{align*}
    \mathcal{L}u(x,t)=\mathrm{p.v.}\int_{\bbR^n}\frac{|u(x,t)-u(y,t)|^{p-2}(u(x,t)-u(y,t))}{|x-y|^{n+sp}}k(x,y,t)\,dy,
\end{align*}
where $s\in(0,1)$, $p
>2$ and $k:\bbR^n\times \bbR^n\times(0,T]\to\bbR$ is a measurable function with 
\begin{align*}
    k(x,y,t)=k(y,x,t)\quad\text{and}\quad \Lambda^{-1}\leq k(x,y,t)\leq \Lambda
\end{align*}
for some constant $\Lambda\geq$1. The aim of the article is to derive modulus of continuity estimates in a quantitative way and to establish the existence of a continuous weak solution to the corresponding initial boundary value problem of the nonlocal two-phase Stefan problem.

\subsection{History and Motivation}
A variational model of phase transition, say from ice to water, as a function of the unknown temperature $\theta$, is given by the following classical Stefan equation
\[
\partial_t (\theta+\beta(\theta)) = \Delta \theta
\]
(see  \cite{stefan1889,kamen61,Ole60}), where $\beta$ is a maximal monotone graph. The reason why we have such a graph is to represent the latent heat of the water molecules as they undergo phase transition from ice at 0\textcelsius~to water at 0\textcelsius; outside of this transition region we have the usual heat equation; we also allow for negative temperatures. The equation also models many other physical phenomena and we refer the interested reader to \cite{vis08}.

The first issue in studying the equation is the notion of solution employed. We are interested in weak solutions which are obtained via integration by parts and differential inclusions. There are other notions of solution possible as in \cite{vis08}. Having established the notion of solution, we turn to the well-posedness for weak solutions. A natural question would be to ask for the existence of a continuous temperature - as ice transitions to water the equation should not predict sudden discontinuities in the temperature. It would also be of interest to quantify the modulus of continuity.

The continuity of temperature was indeed shown in a series of works \cite{diBen82,zie82,cafev83,sac83,diBen86} which also dealt with boundary regularity. Furthermore, the methods in \cite{zie82,diBen82,diBen86} were more general and could deal with a general nonlinear elliptic part with linear growth in the gradient. The paper \cite{diBen86} also provided a boundary modulus of continuity, namely
\begin{equation}\label{eq:mod-1}
    \omega(r) = (\ln|\ln(Cr)|)^{-\alpha} \quad \mbox{ for } \quad 0<r<1 
\end{equation}
where $C>0$ and $\alpha >0$ are universal constants.

For the case where the Laplace operator is replaced by more general degenerate diffusion of $p$-Laplace type ($p>2$), such equations were first studied in  \cite{urb97,urb00} and an interior modulus of continuity was obtained in \cite{BarKuuUrb14} viz.
\begin{equation}\label{eq:mod-2}
    \omega(r) = |\ln(Cr)|^{-\alpha} \quad \mbox{ for } \quad 0<r<1 
\end{equation}
where $C>0$ and $\alpha >0$ are universal constants.
We note that \eqref{eq:mod-2} is an improvement over \eqref{eq:mod-1} even for the Laplace operator. Subsequently, a boundary modulus of continuity was derived in \cite{BarKuuLinUrb18} which was of type \eqref{eq:mod-1}.

Later, the improved modulus \eqref{eq:mod-2} for the boundary case was proved in \cite{Lia22}. 
On the other hand, recently, there are noteworthy results for interior regularity for the Stefan problem with the $p$-Laplace operator. More specifically, when $p=n\geq2$, the improved modulus of continuity, which is
\begin{equation*}
    \omega(r)\eqsim \exp\left(-c|\ln r|^{1/n}\right)
\end{equation*} 
was derived in \cite{Lia24t,GiaLiaUrb24}. Moreover, significant H\"older continuity results with the $p$-Laplace operator when $p>\max\{n,2\}$ were established in \cite{Lia24s,GiaLiaUrb24}, which are notable since such regularity was not readily anticipated for the equation \eqref{eq.main} involving such a degeneracy from $\beta(\cdot)$.

We note that the passage from the Laplace operator to the $p$-Laplace operator for parabolic problems leads to fundamental new difficulties as far as regularity is concerned since the equation is no longer scaling invariant and one must work in a geometry intrinsic to the equation; such a geometry also leads to non-trivial technical issues \cite{urb08}.

When modeling phase transition via the Stefan equation above, one assumes that the material is homogeneous - for example the ice consists only of water molecules. In practice, this corresponds to the assumption that the mean square displacement of the molecule in a given time $t$ is proportional to the time $t$ (which is also reflected in the classical heat equation). However, if strong heterogeneities are present at large scales, such an assumption is untenable. Such situations involving \textit{anomalous} scaling naturally lead to consideration of fractional derivatives \cite{Vol14,Rog23}.

In \cite{Caf22}, it was shown that the Stefan problem involving anomalous diffusion with initial-boundary value conditions admits a continuous solution and a logarithmic modulus of continuity is claimed. On the other hand, we study the initial boundary value problem for diffusion operators comparable to the fractional $p$-Laplace ($p>2$) and construct a weak solution with a quantitative boundary modulus of continuity as in \eqref{eq:mod-2} and also provide quantitative continuity estimates at the initial boundary.

Finally, we note that the regularity results for weak solutions to parabolic nonlocal equations which have no singular effect have been investigated extensively, see for instance \cite{KasWei23,FelKas13,Lia24l,Lia24,AdiPraTew24,ByuKim24,DinZhaZho21,Tav24,BraLinStr21,NguNowSirWei24, Str19,Str19h,PraTew23,Pra24,BanGarKin23,GarLinTav25,DieKimLeeNow25,LiaWei24,TakJiaMen24}.

\subsection{Interior Oscillation Estimate}Now we are going to provide our main results.
First, we investigate a quantitative oscillation estimate of a weak solution to a regularized problem of \eqref{eq.main}. To this end, we fix a cut-off function $\psi\in C_{c}^\infty(-1,1)$ such that $\int_{\bbR}\psi\,dt=1$ and we write $\psi_\epsilon(\xi)\coloneqq\frac{1}\epsilon\psi\left(\frac\xi\epsilon\right)$ for $\epsilon\in(0,1)$. We now consider a weak solution $u_\epsilon$ to
\begin{equation}\label{eq.inta}
    \partial_t(u_\epsilon+\beta_\epsilon(u_\epsilon))+\mathcal{L}u_\epsilon=0\quad\text{in }\Omega\times(0,T],
\end{equation}
where 
\begin{equation}\label{defn.betaeps}
    \beta_\epsilon(\xi)=(\beta*\psi_\epsilon)(\xi).
\end{equation}

Let us give a precise definition of a weak solution to \eqref{eq.inta}.
\begin{definition}
For $\epsilon\in(0,1)$, we say that $u_\epsilon$ is a locally bounded weak solution to \eqref{eq.inta} if 
    \begin{equation*}
        u_\epsilon\in L^{p}_{\mathrm{loc}}(0,T;W^{s,p}_{\mathrm{loc}}(\Omega))\cap C_{\mathrm{loc}}(0,T;L^2(\Omega))\cap L^{\infty}_{\mathrm{loc}}(0,T;L^{p-1}_{sp}(\bbR^n))\cap L^{\infty}_{\mathrm{loc}}(0,T;L^{\infty}_{\loc}(\Omega))
    \end{equation*}
    satisfies 
    \begin{align*}
        &-\int_{t_1}^{t_2}\int_{\Omega}(u_\epsilon+\beta_\epsilon(u_\epsilon))\partial_t\phi\,dx\,dt+\int_{\Omega}(u_\epsilon+\beta_\epsilon(u_\epsilon))\phi\,dx\Bigg\rvert_{t=t_1}^{t=t_2}\\
        &\quad+\int_{t_1}^{t_2}\int_{\bbR^n}\int_{\bbR^n}\frac{|u_{\epsilon}(x,t)-u_{\epsilon}(y,t)|^{p-2}(u_{\epsilon}(x,t)-u_{\epsilon}(y,t))(\phi(x,t)-\phi(y,t))}{|x-y|^{n+sp}}k(x,y,t)\,dx\,dy\,dt=0
    \end{align*}
    for any $\phi\in L^p(t_1,t_2;W^{s,p}(\Omega))\cap W^{1,2}(t_1,t_2;L^2(\Omega))$, where $(t_1,t_2]\Subset (0,T]$ and the support of $\phi(t,\cdot)$ is compactly embedded in $\Omega$ for each $t\in(0,T]$.
\end{definition}

For $f\in L^{\infty}(0,T;L^{p-1}_{sp}(\setR^n))$, $z_0\in\setR^{n+1}$, and $R>0$, we define the tail as follows:
\begin{align*}
\textrm{Tail}(f;Q_{R}(z_0))=\sup_{t_0-\rho^{sp}\leq t\leq t_0}\left(\rho^{sp}\int_{\setR^{n}\setminus B_{\rho}(x_0)}\dfrac{|f(y)|^{p-1}}{|x_0-y|^{n+sp}}\,dy\right)^{\frac{1}{p-1}}.
\end{align*}
For any $z_0=(x_0,t_0)\in\bbR^{n+1}$, $R>0$, $\mathcal{T}$ and $\theta>0$, we define 
\begin{align*}
I^{(\theta)}_{\tau}(t_0)=(t_0-\tau\theta,t_0]\quad\text{and}\quad{Q}^{(\theta)}_{{R},\mathcal{T}}(z_0)\coloneqq B_R(x_0)\times I^{(\theta)}_{\mathcal{T}}(t_0).
\end{align*}
In particular, we write
\begin{align*}
    Q^{(\theta)}_R(z_0)\coloneqq  {Q}^{(\theta)}_{{R},R^{sp}}(z_0)\quad\text{and}\quad Q_R(z_0)\coloneqq  {Q}^{(1)}_{{R},R^{sp}}(z_0).
\end{align*}

We are now ready to state our main results. We first give the explicit interior oscillation estimate.
\begin{theorem}\label{thm.intlast}
    Let $ u_\epsilon$ be a bounded weak solution to 
    \begin{equation*}
        \partial_t(u_\epsilon+\beta_\epsilon(u_\epsilon))+\mathcal{L}u_\epsilon=0\quad\text{in }Q_\mathcal{R}(z_0).
    \end{equation*}
    Let us take 
    \begin{equation*}
        \omega_0\coloneqq \max\{\|u_\epsilon\|_{L^{\infty}({Q}_{\mathcal{R}}(z_0))}+\mathrm{Tail}(u_\epsilon;{Q}_{\mathcal{R}}(z_0)),1\}.
    \end{equation*}
    Suppose $Q_{\rho_0}^{(\omega_0/4)^{2-p}}(z_0)\subset {Q}_{\mathcal{R}}(z_0)$ for some constant $\rho_0\in(0,\mathcal{R}]$. Then there is a constant $\varsigma=\varsigma(n,s,p,\Lambda)\in(0,1)$ such that for any $r< \rho_0$,
    \begin{align}\label{ineq.intlast}
        \osc_{Q^{(\omega_0/4)^{2-p}}_r(z_0)}u_\epsilon\leq c\omega_0\left(1+\ln(\rho_0/r)\right)^{-\varsigma/2}+4\epsilon
    \end{align}
    for some constant $c=c(n,s,p,\Lambda)$ independent of $\epsilon$.
\end{theorem}

\subsection{Oscillation Estimate near the Boundary}
Next, we consider oscillation estimates of solution $u_\epsilon$ to \eqref{eq.inta} up to the boundary. We first recall a fractional Sobolev space. Let $f:\bbR^n\to \bbR$ be a measurable function and $\Omega\Subset \Omega'\subset\bbR^n$. We define
\begin{equation*}
    X_f^{s,p}(\Omega,\Omega')=\{u\in W^{s,p}(\Omega')\,:\,u\equiv f\text{ on }\bbR^n\setminus \Omega\}.
\end{equation*}
By \cite[Lemma 2.11]{BraLinSch18}, we observe $X_0^{s,p}(\Omega,\Omega')=X_0^{s,p}(\Omega,\bbR^n)$. Thus we write 
\begin{equation*}
    X_0^{s,p}(\Omega)\coloneqq X_0^{s,p}(\Omega,\bbR^n).
\end{equation*}
Let us give a notion of weak solution to \eqref{eq.inta} near the boundary.
\begin{definition}
   Let us fix $g\in L^p(0,T;W^{s,p}(\Omega'))\cap L^\infty(0,T;L^{p-1}_{sp}(\bbR^n))$ with $\Omega'\Supset\Omega$. For $I_\mathcal{R}=I_\mathcal{R}(t_0)$, we say that 
    \begin{equation*}
    u_\epsilon\in L^{p}(I_\mathcal{R};W^{s,p}(B_{\mathcal{R}}(x_0)))\cap C(I_{\mathcal{R}};L^2(B_{\mathcal{R}}(x_0)))\cap L^{\infty}(I_\mathcal{R};L^{p-1}_{sp}(\bbR^n))\cap L^{\infty}_{\loc}(I_\mathcal{R};L^{\infty}_{\loc}(\Omega))
    \end{equation*}is a locally bounded weak solution to 
\begin{equation*}
\left\{
\begin{alignedat}{3}
\partial_t (u_\epsilon+\beta_\epsilon(u_\epsilon))+\mathcal{L}u_\epsilon&= 0&&\qquad \mbox{in  $Q_\mathcal{R}(z_0)\cap \Omega_T$}, \\
u_\epsilon&=g&&\qquad  \mbox{in $Q_\mathcal{R}(z_0)\setminus \Omega_T$},
\end{alignedat} \right.
\end{equation*} if $u_\epsilon\in L^{p}(I_\mathcal{R};W^{s,p}(B))$ for some $B\Supset \Omega\cap B_{\mathcal{R}}(x_0)$ and 
    \begin{align*}
        &-\int_{t_1}^{t_2}\int_{\Omega}(u_\epsilon+\beta_\epsilon(u_\epsilon))\partial_t\phi\,dz+\int_{\Omega}(u_\epsilon+\beta_\epsilon(u_\epsilon))\phi\,dx\Bigg\rvert_{t=t_1}^{t=t_2}\\
        &\quad+\int_{t_1}^{t_2}\int_{\bbR^n}\int_{\bbR^n}\frac{|u_{\epsilon}(x,t)-u_{\epsilon}(y,t)|^{p-2}(u_{\epsilon}(x,t)-u_{\epsilon}(y,t))(\phi(x,t)-\phi(y,t))}{|x-y|^{n+sp}}k(x,y,t)\,dx\,dy\,dt=0
    \end{align*}
    holds for any $\phi\in L^p(t_1,t_2;X_0^{s,p}(\Omega\cap B_\mathcal{R}(x_0)))\cap W^{1,2}(t_1,t_2;L^2(\Omega\cap B_\mathcal{R}(x_0)))$ with $(t_1,t_2]\Subset I_R(t_0)$.
\end{definition}

We now assume that the complement of the domain $\Omega$ satisfies a measure density condition. More specifically, there exists a constant $\alpha_0\in(0,1)$ such that for any $x_0\in\partial\Omega$,
    \begin{equation}\label{ass.density}
        \inf_{0<r\leq1}|B_r(x_0)\setminus\Omega|\geq \alpha_0|B_r|.
    \end{equation}
We point out that such an assumption is standard to obtain the continuity at the boundary (see, for instance \cite{KorKuuPal16,BarKuuLinUrb18,Lia22}).

In addition, for any $g\in C(\overline{Q})$ with $Q\subset\setR^{n+1}$ which will be play a role as a boundary data, we always assume 
\begin{equation}\label{cond.oscg}
    |g(z_1)-g(z_2)|\leq \omega_g\left(|x_1-x_2|+|t_1-t_2|^{\frac1{sp}}\right)\quad\text{for any }z_1,z_2\in\overline{Q}
\end{equation}
for some non-decreasing function $\omega_g:\bbR^+\to\bbR^+$ with $\omega_g(0)=0$.

Let us state oscillation estimates at the lateral boundary.
\begin{theorem}\label{thm.bdylast}
For $\epsilon\in(0,1)$ and $\Omega'\Supset\Omega$, let
    \begin{equation*}
        u_\epsilon\in L^p(0,T;W^{s,p}(\Omega'))\cap C(0,T;L^2(\Omega))\cap L^\infty(0,T;L^{p-1}_{sp}(\bbR^n))
    \end{equation*}
    be a bounded weak solution to 
    \begin{equation*}
\left\{
\begin{alignedat}{3}
\partial_t (u_\epsilon+\beta_\epsilon(u_\epsilon))+\mathcal{L}u_\epsilon&= 0&&\qquad \mbox{in  $Q_\mathcal{R}(z_0)\cap \Omega_T$}, \\
u_\epsilon&=g&&\qquad  \mbox{in $Q_\mathcal{R}(z_0)\setminus \Omega_T$},
\end{alignedat} \right.
\end{equation*}
where $x_0\in\partial\Omega$ and $Q_{\mathcal{R}}(z_0)\subset \Omega'_T$.
    Let us assume  
    \begin{equation}\label{ass.bdylast}
        \|u_\epsilon\|_{L^\infty({Q}_{\mathcal{R}}(z_0))}+\mathrm{Tail}(u_\epsilon;{Q}_{\mathcal{R}}(z_0))\leq 1.
    \end{equation}
    Then there exist constants $\varsigma=\varsigma(n,s,p,\Lambda,\alpha_0)\in(0,1)$ such that if
    \begin{align*}
    \omega_g(r)\leq c_g\left(1+\ln(\mathcal{R}/r)\right)^{-\delta}
\end{align*}
holds where $\delta\in(\varsigma,1)$ and $r\in(0,\mathcal{R}]$, then
    \begin{align}\label{ineq.bdylast}
        \osc_{Q_r(z_0)}u_\epsilon\leq c\left(1+\ln(\mathcal{R}/r)\right)^{-\varsigma/2}+4\epsilon
    \end{align}
holds for some constant $c=c(n,s,p,\Lambda,\alpha_0,\delta,c_g)$.
\end{theorem}
We now provide the continuity estimate at the initial boundary.
\begin{theorem}\label{thm.intilast}
For $\epsilon\in(0,1)$ and $\Omega'\Supset\Omega$, let
    \begin{equation*}
        u_\epsilon\in L^p(0,T;W^{s,p}(\Omega'))\cap C(0,T;L^2(\Omega))\cap L^\infty(0,T;L^{p-1}_{sp}(\bbR^n))
    \end{equation*}
    be a bounded weak solution to 
    \begin{equation*}
\left\{
\begin{alignedat}{3}
\partial_t (u_\epsilon+\beta_\epsilon(u_\epsilon))+\mathcal{L}u_\epsilon&= 0&&\qquad \mbox{in  $\left(B_\mathcal{R}(x_0)\times (0,\mathcal{R}^{sp}]\right)\cap \Omega_T$}, \\
u_\epsilon&=g&&\qquad  \mbox{in $(B_\mathcal{R}(x_0)\times [0,\mathcal{R}^{sp}])\setminus \Omega_T$},
\end{alignedat} \right.
\end{equation*}
where $x_0\in\overline{\Omega}$ and $B_\mathcal{R}(x_0)\times (0,\mathcal{R}^{sp}]\subset \Omega'_T$.
    Let us assume  
    \begin{equation}\label{ass.bdyilast}
        \|u_\epsilon\|_{B_\mathcal{R}(x_0)\times [0,\mathcal{R}^{sp}]}+\mathrm{Tail}(u_\epsilon;B_\mathcal{R}(x_0)\times [0,\mathcal{R}^{sp}])\leq 1.
    \end{equation}
    If there are constants $c_g\geq1$ and $\delta\in(0,1)$ such that for any $r\in(0,\mathcal{R}]$,
    \begin{align*}
    \omega_g(r)\leq c_g\left(1+\ln(\mathcal{R}/r)\right)^{-\delta}
\end{align*}
holds, then for any $r\in(0,\mathcal{R}]$,
    \begin{align}\label{ineq.intilast} 
        \osc_{B_r(x_0)\times [0,r^{sp}]}u_\epsilon\leq c\left(1+\ln(\mathcal{R}/r)\right)^{-\delta}
    \end{align}
   holds, where $c=c(n,s,p,\Lambda,\alpha_0,\delta,c_g)$.
\end{theorem}

\begin{remark}
We note that in order to avoid complicated oscillation estimates, we use the normalization conditions \eqref{ass.bdylast} and \eqref{ass.bdyilast}.
    However, by considering the estimates given in Section \ref{sec:4} and Section \ref{sec:5}, we obtain similar quantitative oscillation estimates \eqref{ineq.bdylast} and \eqref{ineq.intilast} without assuming \eqref{ass.bdylast} and \eqref{ass.bdyilast}, respectively.  
\end{remark}

\subsection{Existence of a continuous solution}
Before we state the existence result of the two-phase Stefan problem, we introduce a precise notion of a weak solution to the corresponding initial boundary value problem of \eqref{eq.main}.
\begin{definition}
Let us fix $g\in L^p(0,T;W^{s,p}(\Omega'))\cap L^\infty(0,T;L^{p-1}_{sp}(\bbR^n))$ with $\Omega'\Supset\Omega$.
    We say that  
    \begin{equation*}
        u\in L^p(0,T;W^{s,p}(\Omega))\cap L^\infty(0,T;L^2(\Omega))\cap L^{\infty}(0,T;L^{p-1}_{sp}(\bbR^n))
    \end{equation*}
    is a weak solution to 
    \begin{equation}\label{eq.diri}
\left\{
\begin{alignedat}{3}
\partial_t (u+\beta(u))+\mathcal{L}u&\ni 0&&\qquad \mbox{in  $ \Omega_T$}, \\
u&=g&&\qquad  \mbox{in $\bbR^n\setminus (\partial\Omega\times (0,T]\cup \Omega\times \{t=0\})$},
\end{alignedat} \right.
\end{equation}
 if there is a pair $(u,v)$ with $v\in L^\infty(0,T;L^2(\Omega))$ and 
    \begin{equation*}
        \{v(z)\,:\,z\in \Omega_T\}\subset \{(u+\beta(u))(z)\,:\,z\in\Omega_T\}
    \end{equation*}
    in the sense of graphs and
\begin{align*}
        &-\int_{t_1}^{t_2}\int_{\Omega}v\partial_t \phi\,dz\\
        &\quad+\int_{t_1}^{t_2}\int_{\bbR^n}\int_{\bbR^n}\frac{|u(x,t)-u(y,t)|^{p-2}(u(x,t)-u(y,t))(\phi(x,t)-\phi(y,t))}{|x-y|^{n+sp}}k(x,y,t)\,dx\,dy\,dt\\
        &=-\int_{\Omega}v\phi\,dx\bigg\lvert_{t=t_1}^{t=t_2}
    \end{align*}
    holds for any $(t_1,t_2]\Subset (0,T]$ and
    $\phi\in L^p(t_1,t_2;X_0^{s,p}(\Omega))\cap W^{1,2}(t_1,t_2;L^2(\Omega'))$.
\end{definition}

We now provide the existence result of the two-phase Stefan problem.
\begin{theorem}\label{thm.exist}
    Let us fix $\Omega'\Supset\Omega$ and the complement of the domain $\Omega$ satisfy the measure density condition with $\alpha_0\in(0,1)$. If 
    \begin{equation}\label{cond.g}
        g\in C\left([0,T]\times\overline{\Omega'}\right)\cap L^p(0,T;W^{s,p}(\Omega'))\cap L^\infty(0,T;L^\infty(\bbR^n))
    \end{equation}
    with $\partial_tg\in L^2(\Omega_T)$,
    then there is a weak solution $u\in C\left(\overline{\Omega}\times[0,T]\right)$ to \eqref{eq.diri}. In addition, there is a constant $\varsigma=\varsigma(n,s,p,\Lambda,\alpha_0)$ such that if for any $r>0$,
\begin{align*}
    \omega_g(r)\leq c_g\left(1+|\ln(1/r)|\right)^{-\delta}
\end{align*}
holds for some $\delta\in(\varsigma,1)$, then we have 
\begin{align}\label{ineq.thm.exist}
    \sup_{z_1,z_2\in\overline{\Omega}\times[0,T]}|u(z_1)-u(z_2)|\leq c\left[1+\left|\ln\left(\frac1{|x_1-x_2|+|t_1-t_2|^{\frac1{sp}}}\right)\right|\right]^{-\varsigma/2} 
\end{align}
for some constant $c=c(n,s,p,\Lambda,\alpha_0,\|g\|_{L^\infty},T,\Omega,\Omega',c_g,\delta)$.
\end{theorem}

\subsection{Technical Tools}
Let us give a sketch of the proof of our main theorems. First, we prove uniform oscillation estimates of the solution $u_\epsilon$ to approximated equation \eqref{eq.inta} of \eqref{eq.main}.
To this end, we divide the proof into three parts, interior estimates, lateral boundary estimates and initial boundary estimates.

Based on \cite{BarKuuUrb14,Lia22}, we are going to prove interior oscillation estimates. First, we prove De Giorgi type lemmas under the assumption that the tail is sufficiently small (see Lemma \ref{lem.di1}). Note that the tail does not appear in local problems, so we need to make additional analysis to deal with the term considering appropriate scaling for it. Based on such lemmas, we next construct a sequence of intrinsic cylinders which contain global information about the solution, such as $Q_\rho^{(\omega^{2-p})}$, where $\omega\eqsim \|u_\epsilon\|_{L^\infty(Q)}+\mathrm{Tail}(u_\epsilon;Q)$ (see Section \ref{sec3.1} for more details).
We refer to \cite{AdiPraTew24,Lia24} for constructing such intrinsic cylinders for parabolic fractional $p$-Laplacian type equations. 
With this construction, we are able to obtain that the tail on each cylinder is small enough. Therefore, we apply De Giorgi type lemmas together with two alternative arguments (see \eqref{deg.ass2} and \eqref{eq:mu+u} below), which leads to the desired oscillation estimates.
More specifically, we use an iteration argument to get the following implications
\begin{equation*}
   \osc_{Q_{i}}u_\epsilon\lesssim \omega_i \implies \mathrm{Tail}(u_\epsilon;Q_i)\lesssim \omega_i \implies \osc_{Q_{i+1}}u_\epsilon\lesssim \omega_{i+1}
\end{equation*}
with a sequence of shrinking parabolic cylinders $Q_i$ and non-increasing numbers $\omega_i$, and then employ a technical lemma to derive oscillation estimates.

For the lateral boundary, our method is motivated by \cite{BarKuuLinUrb18,Lia22}. Compared to the interior case, we now consider boundary datum, which makes us use a different type intrinsic cylinder such as $Q_{\rho}^{(\overline{\theta})}$, where $\overline{\theta}\eqsim \omega^{1-p}$. Using the measure density condition as in \eqref{ass.density} together with such a construction of the intrinsic cylinder, we have a measure shrinking lemma  and a De Giorgi type lemma (see Lemma \ref{lem.di4b} and Lemma \ref{lem.di1b}, respectively). As in the interior case, we thus obtain the desired oscillation estimates by constructing an appropriate sequence of intrinsic cylinders and using De Giorgi type lemmas. The remaining proof is about oscillation estimates at the initial level. Since we have energy estimates which is close to elliptic ones (see Lemma \ref{lem.energy3}), there is no singular effect from $\beta_\epsilon$ which is defined in \eqref{defn.betaeps}. Therefore, in a similar way as above, we get the oscillation estimates which do not depend on $\epsilon$.

Combining all three estimates yields uniform global oscillation estimates as in Lemma \ref{lem.globcon}. By the approximation argument as in \cite{BarKuuLinUrb18}, we construct a weak solution to nonlocal two-phase Stefan problem and find a quantitative continuity estimates of such a solution.

\subsection{Organization of the paper}
The paper consists of 6 sections. Section \ref{sec:2} is devoted to basic notation, function spaces, the existence of a regularized equation, and recording necessary lemmas. In Section \ref{sec:3}, we prove interior oscillation estimates, while Sections \ref{sec:4} and \ref{sec:5} are devoted to lateral and initial boundary estimates, respectively. Finally, in Section \ref{sec:6} we prove existence result for our problem.

\section{Preliminaries}\label{sec:2}
We always denote a general constant by $c$, which is equal or bigger than 1. To show the relevant dependencies of constants, we use parentheses, such as $c=c(n,s,p)$, if the constant $c$ depends on $n,s$ and $p$. We will usually denote by $x$ the spatial variable, $t$ the time variable, and $z=(x,t)$ the space-time variable. For an open set $U\subset\setR^n$ and $T>0$, we write
\begin{align*}
U_T:=(0,T]\times U.
\end{align*}
For $f\in L^{1}_{\loc}(I\times U)$ for an open set $U\subset\setR^n$, an open interval $I\subset\setR^n$, and $k\in\setR$, we define the truncation operator as follows:
\begin{align*}
(f-k)_+:=\max\{f-k,0\}\quad\text{and}\quad(f-k)_-:=\max\{k-f,0\}.
\end{align*}

Let a bounded open set $U\subset\mathbb{R}^n$ and $T>0$ be given. We define the function space
\begin{align*}
&C(0,T; L^2_{\text{loc}}(U))\\
&\quad=\left\{f\in L^1(U_T):t\mapsto\|f(\cdot,t)\|_{L^2(K)}\,\,\text{is continuous on }(0,T)\,\,\text{for any compact set }K\Subset U\right\}.
\end{align*}
With $0<s<1\leq p<\infty$, we define the fractional Sobolev space $W^{s,p}(U)$ as 
\begin{align*}
W^{s,p}(U)=\left\{f\in L^p(U)\,\,:\,\,\|f\|_{L^p(U)}+[f]_{W^{s,p}(U)}<\infty\right\}
\end{align*}
equipped with the norm
\begin{align*}
\|f\|_{W^{s,p}(U)}:=\|f\|_{L^p(U)}+[f]_{W^{s,p}(U)}:=\left(\int_{U}|f|^p\,dx\right)^{\frac{1}{p}}+\left(\int_{U}\int_{U}\dfrac{|f(x)-f(y)|^p}{|x-y|^{n+sp}}\,dx\,dy\right)^{\frac{1}{p}}.
\end{align*}
We denote by $W^{s,p}_0(U)$ the space of functions which consists of  $f\in W^{s,p}(\setR^n)$ with $f\equiv 0$ in $\mathbb{R}^n\setminus U$.
Now we define
\begin{align*}
&L^p_{\text{loc}}(0,T; W^{s,p}_{\text{loc}}(U))\\
&\quad=\left\{f\in L^1((0,T]\times U):\text{for any compact }I'\Subset (0,T)\text{ and } K\Subset U,\, \int_{I'}\|f(\cdot,t)\|^p_{W^{s,p}(K)}\,dt<\infty\right\}.
\end{align*}
	
With $s\in(0,1)$ and a time interval $I\subset\mathbb{R}$, and the function space
\begin{align*}
L^{p-1}_{sp}(\mathbb{R}^n)=\left\{f\in L^{1}_{\text{loc}}(\mathbb{R}^n)\,\,:\,\,\int_{\mathbb{R}^n}\dfrac{|f(y)|^{p-1}}{(1+|y|)^{n+sp}}\,dy<\infty\right\},
\end{align*}
we define the following tail space:
\begin{align*}
L^{\infty}(I;L^{p-1}_{sp}(\mathbb{R}^n))=\left\{f\in L^{1}_{\text{loc}}(\mathbb{R}^{n+1})\,\,:\,\,\sup_{t\in I}\int_{\mathbb{R}^n}\dfrac{|f(y,t)|^{p-1}}{(1+|y|)^{n+sp}}\,dy<\infty\right\}.
\end{align*}

We now provide some well-known results about nonlocal parabolic equations and some basic inequalities.

With $\epsilon\in(0,1)$, we recall the regular function $\beta_\epsilon=\beta_\epsilon(\xi)$ defined in \eqref{defn.betaeps}. We directly observe 
\begin{align}\label{cond.betap}
    \beta'_\epsilon\geq0\quad\text{and}\quad \int_{\bbR}\beta'_\epsilon(\xi)\,d\xi=1.
\end{align}
We now verify the existence of a weak solution $u\equiv u_\epsilon$ to the corresponding initial-boundary value problem of a regularized problem of \eqref{eq.main}.
\begin{lemma}
    Let $\Omega'\Supset\Omega$ and
    \begin{equation*}
        g\in L^p(0,T;W^{s,p}(\Omega'))\cap L^{\infty}(0,T;L^{\infty}(\bbR^n))\cap C\left([0,T]\times\overline{\Omega'}\right)\quad\text{and}\quad \partial_t g\in L^{2}(\Omega_T).
    \end{equation*}
    Fix $\epsilon\in(0,1)$. Then there is a unique weak solution 
     \begin{equation*}
         u_\epsilon\in L^{p}(0,T;W^{s,p}(\Omega'))\cap C(0,T;L^2(\Omega))\cap L^\infty(0,T;L^{p-1}_{sp}(\bbR^n))
     \end{equation*} to
     \begin{equation}\label{eq.diri.appro}
\left\{
\begin{alignedat}{3}
\partial_t (u_\epsilon+\beta_\epsilon(u_\epsilon))+\mathcal{L}u_\epsilon&= 0&&\qquad \mbox{in  $ \Omega_T$}, \\
u_\epsilon&=g&&\qquad  \mbox{in $\bbR^n\setminus(\partial\Omega\times (0,T]\cup \Omega\times \{t=0\})$}.
\end{alignedat} \right.
\end{equation}
In addition, we have 
\begin{align}\label{max.exir}
    \sup_{\bbR^n\times [0,T]}|u_\epsilon|\leq \sup_{\bbR^n\times [0,T]}|g|.
\end{align}
\end{lemma}
\begin{proof}
We first observe that $b(\xi)\coloneqq\xi+\beta_\epsilon(\xi)$ is a $C^1$ diffeomorphism with 
\begin{align}\label{diff.b}
    1\leq |b'(\xi)|\leq c=c(\epsilon).
\end{align} Let us define a nonlocal operator
\begin{align*}
    \mathcal{L}_bv(x)=\text{p.v.}\int_{\bbR^n}\frac{\phi(b^{-1}(v)(x)-b^{-1}(v)(y))}{|x-y|^{n+sp}}k(x,y,t)\,dy,
\end{align*}
where $\phi(t)\coloneqq|t|^{p-2}t$. By \eqref{diff.b}, we get 
\begin{align}\label{ineq1.exir}
    \phi(b^{-1}(v)(x)-b^{-1}(v)(y))(v(x)-v(y))\eqsim |v(x)-v(y)|^p
\end{align}
with the implicit constant $c=c(p,\epsilon)$. For each $t>0$, let us write 
\begin{equation*}
    X^{s,p}_{b(g(t))}(\Omega,\Omega')\coloneqq \{f\in W^{s,p}(\Omega')\,:\, f(x)=b(g(x,t))\text{ a.e. }x\in \bbR^n\setminus\Omega\}.
\end{equation*}
We now define for any $v\in X^{s,p}_{b(g(t))}(\Omega,\Omega')$ and $\oldphi \in X_0^{s,p}(\Omega,\Omega')$,
\begin{align*}
\skp{\mathcal{A}_t(v)}{\oldphi}:&=\iint_{\Omega'\times \Omega'}\frac{\phi(b^{-1}(v)(x)-b^{-1}(v)(y))(\oldphi(x)-\oldphi(y))}{|x-y|^{n+sp}}k(x,y,t)\,dx\,dy\\
&\quad+2\int_{\Omega}\int_{\bbR^n\setminus \Omega'}\frac{\phi(b^{-1}(v)(x)-g(y,t))\oldphi(x)}{|x-y|^{n+sp}}k(x,y,t)\,dy\,dx.
\end{align*}
By \eqref{ineq1.exir}, we obtain that the operator $\mathcal{A}_t:X^{s,p}_{b(g(t))}(\Omega,\Omega')\to \left(X_0^{s,p}(\Omega,\Omega')\right)^{*}$ is well-defined. We next define 
\begin{equation*}
\mathcal{A}:X_0^{s,p}(\Omega,\Omega')\times (0,T]\to \left(W^{s,p}(\Omega')\right)^*\quad\text{as}\quad \mathcal{A}(w,t)=\mathcal{A}_t(w+b(g)).
\end{equation*}
We now follow the same lines as in the proof of \cite[Theorem A.3]{BraLinStr21} together with \eqref{diff.b} to see that there is a unique weak solution 
\begin{align*}
    w\in L^{p}(0,T;X^{s,p}_{0}(\Omega,\Omega'))\cap C(0,T;L^2(\Omega))\cap L^\infty(0,T;L^{p-1}_{sp}(\bbR^n))
\end{align*}
to 
\begin{equation*}
\left\{
\begin{alignedat}{3}
\partial_tw+\mathcal{L}_b(w+b(g))&=-\partial_t(b(g))&&\qquad \mbox{in  $ \Omega_T$}, \\
w+b(g)&=b(g)&&\qquad  \mbox{in $\bbR^n\setminus (\partial\Omega\times (0,T]\cup \Omega\times \{t=0\})$}.
\end{alignedat} \right.
\end{equation*}
By considering the function
\begin{align*}
    b^{-1}(w+b(g))\in L^{p}(0,T;W^{s,p}_{g}(\Omega'))\cap C(0,T;L^2(\Omega))\cap L^\infty(0,T;L^{p-1}_{sp}(\bbR^n)),
\end{align*}
which follows from \eqref{diff.b} and the fact that $|b(\xi)|\leq \xi+1$, we obtain a weak solution $u_\epsilon\coloneqq b^{-1}(w+b(g))$ to \eqref{eq.diri.appro}.

We are now going to prove \eqref{max.exir}. We first note from the standard approximation argument as in  the proof of \cite[Proposition A.4]{BraLinStr21} that for any $\tau\in(0,T)$,
\begin{align*}
&\int_{0}^{\tau}\int_{\Omega}\partial_t(b(u_\epsilon))(u_\epsilon-M)_{+}\,dz\\
&\quad=\int_{\Omega}\int_{0}^{u_\epsilon(x,\tau)}b'(\xi)(\xi-M)_+\,d\xi\,dx-\int_{\Omega}\int_{0}^{u_\epsilon(x,0)}b'(\xi)(\xi-M)_+\,d\xi\,dx,
\end{align*}
where $M\coloneqq \|g\|_{L^\infty({\Omega_T})}$. Therefore, by testing $(u_\epsilon-M)_+$ to \eqref{eq.diri.appro} and considering the proof given in \cite[Proposition A.4]{BraLinStr21}, we obtain 
\begin{align*}
    \int_{\Omega}\int_{0}^{u_\epsilon(x,\tau)}b'(\xi)(\xi-M)_+\,d\xi\,dx\leq0.
\end{align*}
Using this together with \eqref{diff.b} and the choice of $\tau\in (0,T)$, we have \eqref{max.exir}. 
\end{proof}

We next provide a basic algebraic lemma which will be used to obtain oscillation estimates. 
\begin{lemma}\label{lem.iter}
    Let the constants $M_2,N_2,L_2\geq4$ be given with $L_2\geq M_2$.
    Let us take $a_n\coloneqq\omega_0^{L_2/M_2}(1+n)^{-\epsilon}$, where $\omega_0\geq1$ and $\epsilon\in(0,1)$. Then there is a constant $\epsilon=\epsilon(M_2,N_2,L_2)$ such that if 
     \begin{align*}
        g(x)\coloneqq1-\frac{|x|^{M_2}}{N_2\omega_0^{L_2}},
    \end{align*}
    then for any $n\geq1$,
    \begin{align}\label{iter.goal}
        a_{n}\geq a_{n-1}g\left(a_{n-1}\right).
    \end{align}   
\end{lemma}
\begin{proof}
    Let us define 
    \begin{align*}
        g_1(x)\coloneqq\left(\frac{x}{1+x}\right)^{\epsilon}\quad\text{and}\quad g_2(x)\coloneqq 1-\frac{1}{M_2N_2x^{\epsilon M_2}}
    \end{align*}
    for any $x\geq0$. 
    We are going to prove that 
    \begin{align*}
        g_1(x)\geq g_2(x)\quad\text{for any }x\geq1.
    \end{align*}
    Then we have 
    \begin{align*}
        \left(\frac{n}{1+n}\right)^{\epsilon}\geq  1-\frac{1}{M_2N_2n^{\epsilon M_2}}\geq 1-\frac{1}{N_2n^{\epsilon M_2}},
    \end{align*}
    which implies \eqref{iter.goal}.
    
   To do this, we first observe 
    \begin{align*}
        g_1'(x)=\frac{\epsilon x^{\epsilon-1}}{(1+x)^{\epsilon+1}}\quad\text{and}\quad g_2'(x)=\frac{\epsilon}{N_2x^{\epsilon M_2+1}}.
    \end{align*}
    By taking 
    \begin{align}\label{iter.ep2}
        \epsilon< 1/(2M_2),
    \end{align}
    we observe that if $x\geq1$, then $\frac{\epsilon}{N_2x^{\frac32}}\leq g_2'(x)$.
    After a few simple computations, we derive
    \begin{align*}
        g'_1(x)\leq \frac{\epsilon}{N_2x^{\frac32}}< g'_2(x)\quad\text{for any }x\geq x_0\coloneqq N_2^2-1.
    \end{align*}
    We next assume
    \begin{equation}\label{iter.ep1}
        \epsilon< \log_2\left(\frac{M_2N_2(N_2^2-1)^{1/2}}{M_2N_2(N_2^2-1)^{1/2}-1}\right)
    \end{equation}
    to see that $(1/2)^\epsilon\geq 1-\frac1{M_2N_2x_0^{\frac12}}$. Therefore, using the fact that $g_1$ and $g_2$ are increasing functions, we obtain that for any $x\in[1,x_0]$,
    \begin{align*}
        g_1(x)\geq g_1(1)\geq 1-\frac1{M_2N_2x_0^{\frac12}}> g_2(x_0)\geq g_2(x).
    \end{align*}
     By \eqref{iter.ep1} and \eqref{iter.ep2}, we fix 
    \begin{align*}
        \epsilon\coloneqq \frac12\min\left\{\frac1{2M_2},\log_2\left(\frac{M_2N_2(N_2^2-1)^{1/2}}{M_2N_2(N_2^2-1)^{1/2}-1}\right)\right\}
    \end{align*}
    to obtain that $(g_1-g_2)(x)>0$ for any $x\in[1,x_0]$ and $(g_1-g_2)'(x)<0$ for any $x\geq x_0$. In addition, we directly observe $\lim_{x\to\infty}(g_1-g_2)(x)=0$. Therefore, we deduce $g_1(x)\geq g_2(x)$ for any $x\geq1$, which completes the proof.
\end{proof}
Finally, we recall a technical lemma (see \cite[Lemma 4.1]{Di93}).
\begin{lemma}\label{lem.tech1}
    Let $A_i$ be a sequence of positive numbers with $A_{i+1}\leq cb^iA_i^{1+\alpha}$, where $c,b\geq1$ and $\alpha>0$. If $A_0\leq c^{-1/\alpha}b^{-1/\alpha^2}$, then $A_i\to0$ as $i\to\infty$.
\end{lemma}

\section{Interior Estimates}\label{sec:3}
Fix $\epsilon\in(0,1)$. In this section, we will investigate a bounded weak solution $u\equiv u_\epsilon$ to 
\begin{align}\label{eq.int.appro}
    \partial_t(u+\beta_\epsilon(u))+\mathcal{L}u=0\quad\text{in }\Omega_T.
\end{align}
Let us take $Q_{R,\mathcal{T}}(z_0)\subset \Omega_T$ for some $R,\mathcal{T}>0$ and $z_0\in\mathbb{R}^{n+1}$.
We first observe a translation invariant and normalization property of \eqref{eq.int.appro}.
\begin{lemma}\label{lem.trans}
For $M>0$, let $u$ be a weak solution to 
    \begin{equation*}
        \partial_t (u+\beta_\epsilon(u))+\mathcal{L}u= 0\quad \mbox{in  $B_R(x_0)\times (t_0-\mathcal{T},t_0]$.}
    \end{equation*}
    Then for $x_0\in\setR^n$ and $t_0\in\setR$, $\overline{u}(x,t)\coloneqq 
     u(x+x_0,t+t_0)/\mathcal{M}$ is a weak solution to
    \begin{equation*}
        \partial_t (\overline{u}+\overline{\beta_\epsilon}(\overline{u}))+\overline{\mathcal{L}}\overline{u}= 0\quad \mbox{in  $B_R\times (-\mathcal{T},0]$,}
    \end{equation*}
    where $\overline{\beta_\epsilon}(\xi)=\frac{\beta_\epsilon(\mathcal{M}\xi)}{\mathcal{M}}$ and
    \begin{align*}
        \overline{L}\overline{u}(x)=\mathrm{p.v.}\int_{\bbR^n}\frac{|\overline{u}(x,t)-\overline{u}(y,t)|^{p-2}(\overline{u}(x,t)-\overline{u}(y,t))}{|x-y|^{n+sp}}\overline{k}(x,y,t)\,dy
    \end{align*}
    with $\mathcal{M}^{p-2}\Lambda^{-1}\leq \overline{k}(x,y,t)\leq \mathcal{M}^{p-2}\Lambda$.
\end{lemma}

We next give Caccioppoli estimates with truncation.
\begin{lemma}\label{lem.energy}
   Let $u$ be a sub(super)-solution to \eqref{eq.int.appro}. Let $\phi$ be a nonnegative smooth function satisfying $\phi(\cdot,t)\equiv0$ on $\bbR^n\setminus B_r(x_0)$ with $r<R$. Then there is a constant $c=c(n,s,p,\Lambda)$ such that
    \begin{align*}
        &\sup_{t_0-\mathcal{T}<t<t_0}\left[\int_{B_R(x_0)}\phi^p(u-k)_{\pm}^{2}\,dx\pm\int_{B_R(x_0)}\phi^p\int_{k}^{u}\beta_\epsilon'(\xi)(\xi-k)_{\pm}\,d\xi\,dx\right]\\
        &\quad+ \int_{t_0-\mathcal{T}}^{t_0}\int_{B_R(x_0)}\int_{B_R(x_0)}\frac{|((u-k)_{\pm}\phi)(x,t)-((u-k)_{\pm}\phi)(y,t)|^p}{|x-y|^{n+sp}}\,dx\,dy\,dt\\
        &\quad+ \int_{t_0-\mathcal{T}}^{t_0}\int_{B_R(x_0)}\int_{B_R(x_0)}\frac{(u-k)_{\mp}^{p-1}(y,t)((u-k)_{\pm}\phi^p)(x,t)}{|x-y|^{n+sp}}\,dx\,dy\,dt\\
        &\leq cR^{p(1-s)}\int_{Q_{R,\mathcal{T}}(z_0)}(u-k)_{\pm}^p|\nabla\phi|^p\,dz
    + c\int_{Q_{R,\mathcal{T}}(z_0)}\left((u-k)_{\pm}^2\pm\int_{k}^{u}\beta_\epsilon'(\xi)(\xi-k)_{\pm}\,d\xi\right)\partial_t \phi^p\,dz\\
        &\quad+c\int_{t_0-\mathcal{T}}^{t_0}\int_{\bbR^n\setminus B_R(x_0)}\int_{B_R(x_0)}\frac{(u-k)_{\pm}^{p-1}(y,t)((u-k)_{\pm}\phi^p)(x,t)}{|x-y|^{n+sp}}\,dx\,dy\,dt\\
        &\quad +\int_{B_R(x_0)\times \{t=t_0-\mathcal{T}\}}\left((u-k)_{\pm}^2\pm\int_{k}^{u}\beta_\epsilon'(\xi)(\xi-k)_{\pm}\,d\xi\right)\phi^p\,dx,
    \end{align*}
    where $k\in\bbR$.
\end{lemma}

\begin{proof}
    Let us take a test function $(u-k)_{\pm}\phi^p$. Then by using \cite[Lemma 2.1]{BarKuuUrb14} to control the time term $\partial_t(u+\beta_\epsilon(u))$, and by using \cite[Proposition 2.1]{Lia24} to handle the space term $\mathcal{L}u$, we deduce the desired estimate.
\end{proof}
By Lemma \ref{lem.energy}, we have energy estimates which remove the role of the singularity $\beta(\cdot)$.
\begin{lemma}\label{lem.nenergy}
    For any $k\geq\epsilon$, we have same energy estimates of $(u-k)_{+}$ as in the parabolic fractional p-Laplacian equations. Similarly, for any $k\leq-\epsilon$, we have same energy estimates of $(u-k)_{-}$ as in the parabolic fractional p-Laplacian equations.
\end{lemma}
\begin{proof}
    When $k\geq\epsilon$ or $k\leq -\epsilon$, then 
    \begin{equation*}
        \int_{k}^{u}\beta_\epsilon'(\xi)(\xi-k)_{\pm}\,d\xi=0,
    \end{equation*}
	since $\beta'_{\epsilon}(\xi)=(\beta\ast\psi'_{\epsilon})(\xi)=0$ when $|\xi|\geq\epsilon$.
    Therefore, we have the same energy estimates as in the ones given in the standard parabolic fractional $p$-Laplacian equations (see \cite{AdiPraTew24,ByuKim24,Lia24}).
\end{proof}

With this energy estimate, we will provide several pointwise estimates. Let us fix $\mathcal{Q}\coloneqq B_R(x_0)\times I\subset \Omega_T$, where $I\coloneqq (t_1-\mathcal{T},t_1]$. We also fix constants $\mu_+, \mu_-$ and $\omega$ as
\begin{align}\label{eq:mu.omega}
    \mu_+\geq \sup_{\mathcal{Q}}u,\,\quad\mu_{-}\leq \inf_{\mathcal{Q}}u,\quad\text{and}\quad\omega\geq\mu_+-\mu_-.
\end{align}
Let us assume $\rho\leq R$.
Using Lemma \ref{lem.energy}, we obtain the following De Giorgi type Lemma.
\begin{lemma}\label{lem.di1}
    Let $u$ be a locally bounded sub(super)-solution to \eqref{eq.int.appro}. Let us fix $\xi\in(0,1/4]$. Suppose 
    \begin{align}\label{tail.ass.di1}
        (\rho/R)^{\frac{sp}{p-1}}\mathrm{Tail}((u-\mu_{\pm})_{\pm};\mathcal{Q})\leq \xi\omega
    \end{align}
    and $Q_\rho^{(\theta)}(z_0)\Subset \mathcal{Q}$, where $\theta=(\xi\omega)^{2-p}$. Then there exists $\nu_0=\nu_0(n,s,p,\Lambda)\in(0,1)$ such that if 
    \begin{align}\label{eq:di1}
        |\{ Q_{\rho}^{(\theta)}(z_0)\,:\pm(\mu_{\pm}-u)\leq \xi\omega\}|\leq \nu_0\left(\frac{\xi\omega}{\max\{1,\xi\omega\}}\right)^{\frac{n+sp}{sp}}|Q_\rho^{(\theta)}(z_0)|,
    \end{align}
    then 
    \begin{equation*}
       \pm(\mu_{\pm}-u)\geq\xi\omega/2\quad\text{in  }Q_{\rho/2}^{(\theta)}(z_0).
    \end{equation*}
\end{lemma}
\begin{proof}
We may assume $z_0=0$ by Lemma \ref{lem.trans}. We first handle the case when $u$ is a super-solution. 
Let us define sequences 
    \begin{equation}\label{seq.di1}
        k_i\coloneqq \mu_-+\frac{\xi\omega}2+\frac{\xi\omega}{2^{i+1}}
        ,\quad\rho_i\coloneqq\frac{\rho}2+\frac{\rho}{2^{i+1}}\quad\text{and}\quad\overline{\rho}_i\coloneqq\frac{\rho_i+\rho_{i+1}}2
    \end{equation}
    for any $i\geq 0$. We next define 
    \begin{equation*}
        A_i\coloneqq\frac{|\{Q_{\rho_i}^{(\theta)}\,:\,u<k_i\}|}{|Q_{\rho_i}^{(\theta)}|}
    \end{equation*}
    to claim that 
    \begin{equation}\label{ineq0.di1}
        A_{i+1}\leq cM^{i}\frac{\max\{1,\xi\omega\}^{1+sp/n}}{(\xi\omega)^{1+sp/n}}A_i^{1+sp/n}
    \end{equation}
    for some constant $c=c(n,s,p,\Lambda)\geq 1$ and $M=M(n,s,p,\Lambda)\geq 1$. To obtain \eqref{ineq0.di1}, we first observe from \cite[Lemma 2.3]{DinZhaZho21} that
    \begin{equation}\label{ineq01.di1}
    \begin{aligned}
       &\dashint_{Q^{(\theta)}_{\rho_{i+1}}}{(u-{k}_{i})^{p(1+2s/n)}_-}\,dz\\
       &\leq c\left(\dashint_{Q^{(\theta)}_{\rho_{i+1}}}\int_{B_{\rho_{i+1}}}\rho_i^{sp}\frac{|(u-{k}_{i})_-(x,t)-(u-{k}_{i})_-(y,t)|^p}{|x-y|^{n+sp}}\,dy\,dz+\dashint_{Q^{(\theta)}_{\rho_{i+1}}}(u-{k}_{i})^p_-\,dz\right)\\
       &\quad\times \left(\sup_{t\in I^{({\theta})}_{\rho_{i+1}}}\dashint_{B_{\rho_{i+1}}}(u-{k}_{i})^2_-\,dx\right)^{\frac{sp}{n}}\eqqcolon J_1\times J_2
    \end{aligned}
    \end{equation}
    for some constant $c=c(n,s,p)$. We will apply Lemma \ref{lem.energy} to \eqref{ineq01.di1}. To do this, let
    us choose a nonnegative cut-off function $\psi\in C_c^{\infty}(B_{\overline{\rho}_i})$ with $\psi\equiv1$ in $B_{\rho_{i+1}}$ and $\abs{\nabla\psi}\leq c(n)2^i/\rho$, and $\oldphi\in C^\infty(\bbR)$ with $\oldphi\equiv1$ on $t\geq -\rho_{i+1}^{sp}{\theta}$ and $\oldphi\equiv0$ on $t<-\overline{\rho}_{i}^{sp}{\theta}$.
    For $\phi=\psi\oldphi$, we compute 
    \begin{equation}\label{ineq02.di1}
    \begin{aligned}
        &\rho^{p}\dashint_{Q_{\rho_i}^{(\theta)}}(u-{k}_i)^{p}_{-}|\nabla \phi|^p\,dz+c\rho^{sp}\dashint_{Q_{\rho_i}^{(\theta)}}\left((u-{k}_i)_{-}^2+\int_{u}^{{k}_i}\beta_\epsilon'(\xi)(\xi-{k}_i)_{-}\,d\xi\right)\partial_t \phi^p\,dz\\
        &\leq c\left[2^{ip}\dashint_{Q_{\rho_i}^{(\theta)}}(u-{k}_i)^{p}_{-}\,dz+\frac{2^i}{\theta}\dashint_{Q_{\rho_i}^{(\theta)}}(u-{k}_i)_{-}^2+(u-{k}_i)_{-}\,dz\right]
    \end{aligned}
    \end{equation}
    for some constant $c=c(n,s,p)$, where we have used \eqref{cond.betap}.
    In light of the fact that 
    \begin{equation*}
        |y-x|\geq |y|/2^{i+1}\quad\text{for any }x\in B_{\overline{\rho}_i}\text{ and }y\in \bbR^n\setminus B_{\rho_i}
    \end{equation*}
    together with \eqref{tail.ass.di1}, we next obtain
    \begin{equation}\label{ineq03.di1}
    \begin{aligned}
        &\rho^{sp}\dashint_{I_{\rho_i}^{(\theta)}}\int_{\bbR^n\setminus B_{\rho_i}}\dashint_{B_{\rho_i}}\frac{(u-{k}_i)^{p-1}_{-}(y,t)\left((u-{k}_i)_{-}\phi^p\right)(x,t)}{|x-y|^{n+sp}}\,dx\,dy\,dt\\
        &\leq c\rho^{sp}2^{i(n+sp)}\dashint_{I_{\rho_i}^{(\theta)}}\int_{\bbR^n\setminus B_{\rho_i}}\dashint_{B_{\rho_i}}\frac{(u-{k}_i)^{p-1}_{-}(y,t)\left((u-{k}_i)_{-}\phi^p\right)(x,t)}{|y|^{n+sp}}\,dx\,dy\,dt\\
        &\leq c2^{i(n+sp)}(\xi\omega)^{p-1}\dashint_{Q_{\rho_i}^{(\theta)}}(u-{k}_i)_{-}\,dz.
    \end{aligned}
    \end{equation}
    By using \eqref{ineq02.di1}, \eqref{ineq03.di1}, ${\theta}=(\xi\omega)^{2-p}$, and Lemma \ref{lem.energy} with $k={k}_i$, $\phi=\psi\oldphi$ and $Q_{R,\mathcal{T}}(z_0)=Q_{\rho_i}^{(\theta)}$,
    \begin{align}\label{ineq1.di1}
        J_1&\leq c\left[\dashint_{Q_{\rho_i}^{({\theta})}}2^{i(n+sp+p)}(u-{k}_{i})_-^p+
    \frac{2^{i(n+sp)}}{{\theta}}\left((u-{k}_{i})_-^2+(1+\xi\omega)(u-{k}_{i})_-\right)\,dz\right]
    \end{align}
    holds and
    \begin{align}\label{ineq2.di1}
        (\rho_{i+1}^{sp}{\theta})^{-1}J_2^{\frac{n}{sp}}\leq c\rho_i^{-sp}\left[\dashint_{Q_{\rho_i}^{({\theta})}}2^{i(n+sp+p)}(u-{k}_{i})_-^p+\frac{2^{i(n+sp)}}{{\theta}}\left((u-{k}_{i})_-^2+(1+\xi\omega)(u-{k}_{i})_-\right)\,dz\right]
    \end{align}
    for some constant $c=c(n,s,p,\Lambda)$.
    By the choice of the constant ${\theta}=(\xi\omega)^{2-p}$, and $(u-k_i)_-\leq \xi\omega$ from \eqref{seq.di1}, we have 
    \begin{align*}
        \dashint_{Q_{\rho_i}^{({\theta})}}{\theta}^{-1}\left((u-{k}_{i})_-^2+(u-{k}_{i})_-\right)\,dz\leq c(\xi\omega)^{p-2}\left((\xi\omega)^{2}+(\xi\omega)\right)A_i.
    \end{align*}
    Plugging this into \eqref{ineq1.di1} and \eqref{ineq2.di1} with $(u-k_i)_-\leq \xi\omega$ again, we obtain
    \begin{align*}
        J_1\leq c2^{i(n+sp+p)}\max\{1,\xi\omega\}(\xi\omega)^{p-1}A_i
    \end{align*}
    and with $\theta=(\xi\omega)^{2-p}$,
    \begin{align*}
    J_2\leq c\left({\theta}2^{i(n+sp+p)}\max\{1,\xi\omega\}(\xi\omega)^{p-1}A_i\right)^{\frac{sp}n}\leq c\left(2^{i(n+sp+p)}\max\{1,\xi\omega\}(\xi\omega)A_i\right)^{\frac{sp}n}
    \end{align*}
    for some constant $c=c(n,s,p,\Lambda)$, which implies 
    \begin{align*}
        \dashint_{Q^{(\theta)}_{\rho_{i+1}}}{(u-{k}_{i})^{p(1+2s/n)}_-}\,dz\leq c2^{i(n+sp+p)(1+sp/n)}\max\{1,\xi\omega\}^{1+sp/n}(\xi\omega)^{p-1+sp/n}A_i^{1+sp/n}
    \end{align*}
together with \eqref{ineq01.di1}. Therefore, we have 
    \begin{equation}\label{ineq3.di1}
    \begin{aligned}
        A_{i+1}&\leq \frac{c2^{ip(1+2s/n)}}{(\xi\omega)^{p(1+2s/n)}}\dashint_{Q^{(\theta)}_{\rho_{i+1}}}{(u-{k}_{i})^{p(1+2s/n)}_-}\,dz\\
        &\leq {c2^{2ip(n+sp+p)(1+sp/n)}}\frac{\max\{1,\xi\omega\}^{1+sp/n}}{(\xi\omega)^{1+sp/n}}A_i^{1+sp/n}
    \end{aligned}
    \end{equation}
    for some constant $c=c(n,s,p,\Lambda)\geq 1$, which proves \eqref{ineq0.di1} with $M=2^{2p(n+sp+p)(1+sp/n)}\geq 1$. By applying Lemma \ref{lem.tech1} into \eqref{ineq3.di1}, we deduce the desired result for some $\nu_0=\nu_0(n,s,p,\Lambda)\in(0,1)$. Similarly, we also get the conclusion when $u$ is a sub-solution.    
\end{proof}
We give a remark for a similar version of Lemma \ref{lem.di1} when the singular $\beta(\cdot)$ has no effect.
\begin{remark}\label{rmk.nondeg}
    Suppose $\pm(\mu_{\pm}-\epsilon)\geq\omega/4$. By Lemma \ref{lem.nenergy} and \cite[Lemma 3.1]{Lia24}, we obtain that if
    \begin{align*}
        |\{ Q_{\rho}^{(\theta)}(z_0)\,:\pm(\mu_{\pm}-u)\leq \xi\omega\}|\leq \nu_0|Q_\rho^{(\theta)}|\quad\text{and}\quad( \rho/R)^{\frac{sp}{p-1}}\mathrm{Tail}((u-\mu_{\pm})_{\pm};\mathcal{Q})\leq \xi\omega,
    \end{align*}
    then 
    \begin{equation*}
       \pm(\mu_{\pm}-u)\geq\xi\omega/2\quad\text{in  }Q_{\rho/2}^{(\theta)}(z_0).
    \end{equation*}
\end{remark}
We next prove a variant of De Giorgi type lemma.
\begin{lemma}\label{lem.di2}
    Let $u$ be a locally bounded sub(super)-solution to \eqref{eq.int.appro} and let $\xi\in(0,1/4]$. Then there is a constant $\nu_d=\nu_d(n,s,p,\Lambda)\in(0,1)$ such that if 
\begin{equation}\label{ass.di2}
    \pm(\mu_{\pm}-u(\cdot,t_1))\geq \xi\omega\quad\text{in }B_\rho(x_0),
\end{equation}
then for $\Theta=(\xi\omega)^{2-p}$,
\begin{equation*}
    \pm(\mu_{\pm}-u)\geq \xi\omega/2\quad\text{in }B_{\rho/2}(x_0)\times(t_1,t_1+\nu_d\Theta\rho^{sp}],
\end{equation*}
whenever 
\begin{align*}
       (\rho/R)^{\frac{sp}{p-1}} \mathrm{Tail}((u-\mu_{\pm})_{\pm};\mathcal{Q})\leq \xi\omega
    \end{align*}
    and $B_\rho(x_0)\times (t_1,t_1+\nu_d\theta\rho^{sp}]\subset \mathcal{Q}$.
\end{lemma}
\begin{proof}
    By \eqref{ass.di2}, we have 
    \begin{equation*}
    \int_{B_\rho(x_0)\times \{t=t_1\}}\left((u-k)_{\pm}^2\pm\int_{k}^{u}\beta_\epsilon'(\xi)(\xi-k)_{\pm}\,d\xi\right)\phi^p\,dx=0
    \end{equation*}
    if $k=\mu_{\pm}\mp\xi\omega$.
    Therefore, by following the same lines as in the proof of \cite[Lemma 3.2]{Lia24}, we get the desired estimate.
\end{proof}

We provide more lemmas in case of $\pm(\mu_{\pm}\mp\omega/4)\geq \epsilon$. Since in this case, there is no effect of the singularity $\beta(\cdot)$, the following two lemmas are obtained by \cite[Lemma 3.3 and Lemma 3.4]{Lia24}.
\begin{lemma}\label{lem.di3}
    Let $u$ be a locally bounded sub(super)-solution to \eqref{eq.int.appro} with $\pm(\mu_{\pm}\mp\omega/4)\geq \epsilon$ and let $\alpha,\xi\in(0,1)$ and $\xi\in(0,1)$ be given. Then there are constants $\delta,\tau\in(0,1)$ depending only on $n,s,p,\Lambda$, and $\alpha$, but independent of $\xi$, such that if 
    \begin{equation*}
        |\{\pm(\mu_{\pm}-u(\cdot,t_1))\geq \xi\omega\}\cap B_{\rho}(x_0)|\geq \alpha|B_\rho|,
    \end{equation*}
    then 
    \begin{equation*}
        |\{\pm(\mu_{\pm}-u(\cdot,t))\geq \tau\xi\omega\}\cap B_{\rho}(x_0)|\geq \alpha/2|B_\rho|\quad\text{for any }t\in (t_1,t_1+\delta\theta\rho^{sp}]
    \end{equation*}
    whenever 
    \begin{equation*}
        (\rho/R)^{\frac{sp}{p-1}}\mathrm{Tail}((u-\mu_\pm)_\pm;\mathcal{Q})\leq \xi\omega
    \end{equation*}
    and $B_{\rho}(x_0)\times (t_1,t_1+\delta\theta\rho^sp]\subset\mathcal{Q}$,
    where $\theta=(\xi\omega)^{2-p}$. In particular, we observe $\tau=\frac{\alpha}{c}$ and $\delta=\frac{\alpha^{n+p+1}}{c}$ for some constant $c=c(n,s,p,\Lambda)$.
\end{lemma}

\begin{lemma}\label{lem.di4}
    Let $u$ be a locally bounded sub(super)-solution to \eqref{eq.int.appro} with $\pm(\mu_{\pm}\mp\omega/4)\geq \epsilon$. Let us fix $\xi\in(0,1/4]$, $\varsigma\in(0,1]$ and write $\theta=(\varsigma\xi\omega)^{2-p}$. If there is a constant $\alpha\in(0,1)$ such that 
    \begin{equation*}
        |\{\pm(\mu_{\pm}-u(\cdot,t))\geq\xi\omega\}\cap B_\rho(x_0)|\geq\alpha|B_\rho|\quad\text{for any }t\in(t_0-\theta\rho^{sp},t_0]
    \end{equation*}
    and
    \begin{equation*}
        (\rho/R)^{\frac{sp}{p-1}}\mathrm{Tail}((u-\mu_\pm)_\pm;\mathcal{Q})\leq \varsigma\xi\omega
    \end{equation*}
    with $Q_{\rho}^{(\theta)}(z_0)\subset \mathcal{Q}$, then for some $c=c(n,s,p,\Lambda)$, we have
    \begin{equation*}
        \left|\{\pm(\mu_{\pm}-u)\leq\varsigma\xi\omega/4\}\cap Q^{(\theta)}_\rho(z_0)\right|\leq \frac{c\varsigma^{p-1}}{\alpha}|Q^{(\theta)}_\rho|.
    \end{equation*}
\end{lemma}

\subsection{Oscillation estimate of solutions}\label{sec3.1}
In this subsection, we prove Theorem \ref{thm.intlast}. Before we give the proof, we first construct a sequence of cylinders $Q_i$ defined in \eqref{seq.cylinder} and find suitable oscillation estimates on these cylinders (see \eqref{ind.int}).

Let us fix $Q_\mathcal{R}(z_0)\subset \Omega_T$ and 
\begin{equation*}
    \omega_0\coloneqq \max\left\{\|u\|_{L^\infty(Q_\mathcal{R}(z_0))}+\mathrm{Tail}(u;Q_\mathcal{R}(z_0)),1\right\}.
\end{equation*}
Choose $\rho_0\leq\mathcal{R}$ such that
\begin{equation*}
    Q_{\rho_0}^{(\omega_0/4)^{2-p}}(z_0)\subset Q_{\mathcal{R}}(z_0).
\end{equation*}
We define 
\begin{equation}\label{defn.f}
    f_1(x)\coloneqq \frac{|x|^{M_1}}{N_1\omega_0^{M_1}}\quad\text{and}\quad f_2(x)=1-\frac{|x|^{M_2}}{N_2\omega_0^{M_2}},
\end{equation}
where $M_1,M_2,N_1,N_2\geq4$ are the constants determined later. 
Take sequences
\begin{equation}\label{seq.omega}
    \rho_{i+1}\coloneqq {f_1}(\omega_{i})\rho_i\quad\text{and}\quad \omega_{i+1}\coloneqq\max\left\{\omega_if_2(\omega_i),\omega_i/2^{s},4\epsilon\right\},
\end{equation}
where  
\begin{equation}\label{seq.cylinder}
    Q_i\coloneqq Q^{(\theta_i)}_{\rho_i}(z_0),\quad\theta_i\coloneqq(\omega_i/4)^{2-p},\quad\mu_{i}^+\coloneqq \sup_{Q_{i}}u\quad\text{and}\quad \mu_{i}^-\coloneqq \sup_{Q_{i}}u-\omega_i.
\end{equation}
We are now going to determine the constants $M_1,M_2,N_1$ and $N_2$ to see that
\begin{equation}\label{ind.int}
    \osc_{Q_i}u\leq \omega_i\quad\text{and}\quad Q_{i}\subset Q_{i-1} \quad\text{for each }i\geq0\quad(i\in\setN),
\end{equation}
where $Q_{-1}=\mathcal{Q}_{\mathcal{R}}(z_0)$.
Suppose that 
\begin{equation*}
    \osc_{Q_i}u\leq \omega_i\quad\text{and}\quad Q_{i}\subset Q_{i-1} 
\end{equation*}
hold for any $i=0,1,\ldots j-1$.
In this setting, we first observe that there is a constant $c_0=c_0(n,s,p)\geq 1$ which is independent of $i$ such that for any $i=0,1,\ldots j-1$,
\begin{equation}\label{ineq.tail}
     \mathrm{Tail}\left((u-\mu^{\pm}_{i})_{\pm};Q_i\right)\leq c_0\omega_i.
\end{equation}

To this end, we first see
\begin{align*}
\mathrm{Tail}\left((u-\mu^{\pm}_{i})_{\pm};Q_i\right)&\leq c
\sum_{k=1}^{i}(\rho_i/\rho_k)^{\frac{sp}{p-1}}\sup_{t\in I_i}\left(\dashint_{B_{k}\setminus B_{k-1}}(u-\mu^{\pm}_{i})_{\pm}^{p-1}\,dx\right)^{\frac1{p-1}}\\
&\quad+c(\rho_i/\rho_0)^{\frac{sp}{p-1}}\sup_{t\in I_i}\left(\rho_{0}^{sp}\int_{\bbR^n\setminus B_{0}}\frac{(u-\mu^{\pm}_{i})_{\pm}^{p-1}}{|y|^{n+sp}}\,dy\right)^{\frac1{p-1}} \coloneqq J_1+J_2,
\end{align*}
where we write 
\begin{equation}\label{eq:BIQ}
B_i\coloneqq B_{\rho_i},\quad I_i\coloneqq I_{\rho_i^{sp}}^{(\omega_i/4)^{2-p}},\quad \text{and}\quad Q_i:=B_i\times I_i.
\end{equation}
Since $f_1(\omega_i)\leq 1/2$ (from $M_1,N_1\geq 4$ and $\omega_i\leq \omega_0$) and $\omega_{i+1}\geq \omega_i/2^{s}$ by \eqref{seq.omega}, we estimate $J_1$ as 
\begin{align*}
    J_1\leq c\sum_{k=1}^i 2^{-(i-k)\frac{sp}{p-1}}\omega_{k}\leq c\sum_{k=1}^i 2^{-(i-k)\frac{sp}{p-1}}2^{s(i-k)}\omega_i\leq c\sum_{k=1}^i 2^{-(i-k)\frac{s}{p-1}}\omega_{i}\leq c\omega_i,
\end{align*}
where $c=c(n,s,p)$. Again $\omega_{i+1}\geq\omega_{i}/2^{s}$ by \eqref{seq.omega}, we estimate $J_2$ as $J_2\leq c2^{-\frac{isp}{p-1}}\omega_0\leq c\omega_i$, where $c=c(n,s,p)$. Combining the estimates $J_1$ and $J_2$ yields \eqref{ineq.tail}.
For convenience, we write 
\begin{equation}\label{eq:RQw}
R\coloneqq \rho_{j-1},\quad \mathcal{Q}\coloneqq Q_{j-1},\quad\omega\coloneqq \omega_{j-1},\quad \mu_+\coloneqq \mu_{j-1}^{+}\quad\text{and}\quad \mu_-\coloneqq \mu_{j-1}^{-}.
\end{equation}
In addition, by Lemma \ref{lem.trans}, we may assume $z_0=0$.

By the choice of $\mu_+,\mu_-$, and $\omega$ given in \eqref{seq.cylinder}, either
\begin{equation*}
    \mu_+-\omega/4\geq \omega/4\quad\text{or}\quad \mu_{-}+\omega/4\leq-\omega/4
\end{equation*}
holds by the definition of $\omega_j$. It suffices to show the case when $\mu_+-\omega/4\geq \omega/4$, as we can similarly prove the case that $\mu_{-}+\omega/4\leq-\omega/4$.
Therefore, $\mu_+-\omega/4\geq \epsilon$ by \eqref{seq.omega}.

Let $\gamma\geq2$ to be determined later depending only on $n,s,p$ and $\Lambda$ (see \eqref{condf.gamma}). We next let $\sigma\in(0,1]$ which will be chosen later depending only on $n,s,p$ and $\omega$ (see \eqref{condf.sigma}). We write
\begin{equation}\label{defn.zeta}
    \zeta\coloneqq 1/(p-1),\quad \rho\coloneqq \sigma R/\gamma\quad\text{and}\quad r\coloneqq \sigma^{\zeta }R/\gamma.
\end{equation}
Since $\sigma\leq1$ and $\zeta<1$, we have $\rho\leq r$.

\textbf{First alternative :}
We now assume that there is a time level 
\begin{align}\label{eq:timelevel}
t\in (-r^{sp}(\omega/4)^{2-p},0]
\end{align}
such that 
\begin{equation}\label{deg.ass2}
    \left|\left\{Q_{\rho}^{(\theta)}(0,t)\,:\,u-\mu_-\leq\omega/4\right\}\right|\leq \nu_0\overline{\omega}\left|Q_{\rho}^{(\theta)}(0,t)\right|,
\end{equation}
where the constant $\nu_0=\nu_0(n,s,p,\Lambda)$ is determined in Lemma \ref{lem.di1}, and
\begin{equation}\label{choi.thevome}
    \overline{\omega}\coloneqq  \left(\frac{\omega/4}{\max\{1,\omega/4\}}\right)^{\frac{n+sp}{sp}}\quad\text{and}\quad \theta=\left(\frac{\omega}{4}\right)^{2-p}.
\end{equation}
We now choose 
\begin{equation}\label{cond1.gamma}
    \gamma\geq \left(1+2^{\frac{p-1}{sp}}\right)(4c_0)^{\frac{p-1}{sp}}
\end{equation}
to see that
\begin{align}\label{eq:subsetQ}
Q_\rho^{(\theta)}(0,t)\subset \mathcal{Q}.
\end{align}
Indeed, by \eqref{defn.zeta} and \eqref{choi.thevome}, $Q^{(\theta)}_{\rho}(0,t)=B_{\rho}(0)\times (t-\theta\rho^{sp},t]=B_{\frac{\sigma R}{\gamma}}(0)\times (t-(\tfrac{\omega}{4})^{2-p}(\tfrac{\sigma R}{\gamma})^{sp},t]$ holds. Also, by \eqref{eq:timelevel} and \eqref{eq:RQw}, we obtain
\begin{align*}
\mathcal{Q}=Q_{i-1}=B_{\rho_{i-1}}(0)\times I^{(\omega_{i-1}/4)^{2-p}}_{\rho^{sp}_{i-1}}=B_{R}(0)\times I^{(\omega/4)^{2-p}}_{R^{sp}}=B_{R}(0)\times(-(\tfrac{\omega}{4})^{2-p}R^{sp},0).
\end{align*}
Then if $\gamma\geq 1+2^{\frac{1}{sp}}$, since $\sigma\leq 1$, we have $B_{\frac{\sigma R}{\gamma}}(0)\subset B_{R}(0)$. Moreover, by $t>-r^{sp}\left(\frac{\omega}{4}\right)^{2-p}$, \eqref{defn.zeta}, and $\zeta<1$, there holds
\begin{align*}
t-\left(\frac{\omega}{4}\right)^{2-p}\left(\frac{\sigma R}{\gamma}\right)^{sp}&\geq\left(\frac{\omega}{4}\right)^{2-p}\left[-r^{sp}-\left(\frac{\sigma R}{\gamma}\right)^{sp}\right]\\
&\geq\left(\frac{\omega}{4}\right)^{2-p}\left[-\left(\frac{\sigma^{\zeta}R}{\gamma}\right)^{sp}-\left(\frac{\sigma R}{\gamma}\right)^{sp}\right]\geq-\left(\frac{\omega}{4}\right)^{2-p}R^{sp}.
\end{align*}
Thus $(t-(\frac{\omega}{4})^{2-p}(\frac{\sigma R}{\gamma})^{sp},t)\subset (-(\frac{\omega}{4})^{2-p}R^{sp},0)$ so that \eqref{eq:subsetQ} is proved.

Also, the choice of $\gamma$ in \eqref{cond1.gamma} helps us to estimate
\begin{equation*}
    (\rho/R)^{\frac{sp}{p-1}}\mathrm{Tail}((u-\mu_-)_{-};\mathcal{Q})=(\sigma/\gamma)^{\frac{sp}{p-1}}\mathrm{Tail}((u-\mu_-)_{-};\mathcal{Q})\leq \omega/4,
\end{equation*}
which follows from \eqref{defn.zeta} and \eqref{ineq.tail}.
Then by Lemma \ref{lem.di1}, we get 
\begin{equation}\label{deg.res1}
    u-\mu_-\geq \omega/8\quad\text{in }Q^{(\theta)}_{\rho/2}(0,t).
\end{equation}
Hence, there is a constant
\begin{equation}\label{deg.xi0}
    \xi_1\geq\sigma^{\frac{sp}{p-1}}\nu_d^{\frac1{p-2}}2^{-\frac{sp}{p-2}}\coloneqq 
    \xi_0(\leq 1)
\end{equation}
such that
\begin{equation}\label{deg.time}
    t+\nu_d(\xi_1\omega/4)^{2-p}(\rho/2)^{sp}=0,
\end{equation}
where $\nu_d=\nu_d(n,s,p,\Lambda)\in(0,2^{-(p-2)}]$ is determined in Lemma \ref{lem.di2}. Here, note that
\begin{align*}
\nu_d(\xi_0\omega/4)^{2-p}(\rho/2)^{sp}&=\nu_d\left(\sigma^{\frac{sp}{p-1}}\nu_d^{\frac{1}{p-2}}2^{-\frac{sp}{p-2}}\frac{\omega}{4}\right)^{2-p}\left(\frac{\rho}{2}\right)^{sp}=\left(\frac{\sigma^{\zeta}\rho}{\sigma}\right)^{sp}\left(\frac{\omega}{4}\right)^{2-p}=r^{sp}\left(\frac{\omega}{4}\right)^{2-p}
\end{align*}
by \eqref{defn.zeta} and so $t+\nu_d(\xi_0\omega/4)^{2-p}(\rho/2)^{sp}>0$ by \eqref{eq:timelevel}. We now select 
\begin{equation}\label{cond2.gamma}
    \gamma\geq 2^{\frac{p-1}{p-2}}\left({8c_0}\nu_d^{-\frac1{p-2}}\right)^{\frac{p-1}{sp}}
\end{equation}
to see that
\begin{equation}\label{deg.ass4}
\begin{aligned}
    (\rho/R)^{\frac{sp}{p-1}}\mathrm{Tail}((u-\mu_-)_{-};\mathcal{Q})=(\sigma/\gamma)^{\frac{sp}{p-1}}\mathrm{Tail}((u-\mu_-)_{-};\mathcal{Q})&\leq \sigma^{\frac{sp}{p-1}} \nu_d^{\frac1{p-2}}2^{-\frac{sp}{p-2}}\omega/8\\
    &\leq \xi_0\omega/8\leq \xi_1\omega/8,
\end{aligned}
\end{equation}
where we have used \eqref{ineq.tail} and \eqref{deg.xi0}.
We first consider the case when $\xi_1\leq1$.
Since \eqref{deg.res1} implies
\begin{equation}\label{ineq.000}
u(\cdot,\tau)-\mu_{-} \geq\omega/8\geq\xi_1\omega/8\quad\text{in }B_{\rho/2} \quad\text{for any }\tau\in(t-(\rho/2)^{sp}\theta,t],
\end{equation}
by Lemma \ref{lem.di2} for $\Theta=(\xi_1\omega/8)^{2-p}$ together with \eqref{deg.time} and \eqref{deg.ass4}, we have 
\begin{equation}\label{deg.res2}
    u-\mu_-\geq \xi_1\omega/16\quad\text{in }B_{\rho/4}\times (-(\rho/2)^{sp}\theta,0].
\end{equation}
On the other hand, if $\xi_1\geq1$, then from \eqref{deg.time}, $t\in [-\nu_d(\omega/4)^{2-p}(\rho/2)^{sp},0]$. Thus there is a time level $\widetilde{t}\in(t-(\omega/4)^{2-p}(\rho/2)^{sp},t]$ such that $\widetilde{t}=-\nu_d(\omega/8)^{2-p}(\rho/2)^{sp}$. Hence, from \eqref{deg.res1} we have 
\begin{align*}
u(\cdot,\widetilde{t})-\mu_-\geq \omega/8\quad\text{in }B_{\rho/2}.
\end{align*}
We now use \eqref{deg.ass4} and $\xi_0\leq 1$ to see that
\begin{align*}
    (\rho/R)^{\frac{sp}{p-1}}\mathrm{Tail}((u-\mu_-)_{-};\mathcal{Q})=(\sigma/\gamma)^{\frac{sp}{p-1}}\mathrm{Tail}((u-\mu_-)_{-};\mathcal{Q})\leq \xi_0\omega/8\leq \omega/8.
\end{align*}
By Lemma \ref{lem.di2}, we get 
\begin{equation}\label{deg.res21}
    u-\mu_-\geq \omega/16\quad\text{in }B_{\rho/4}\times (-(\rho/2)^{sp}\theta,0].
\end{equation}
In light of \eqref{deg.res1}, \eqref{deg.res2} and \eqref{deg.res21} along with \eqref{deg.xi0}, we get
\begin{align}\label{deg.res.fir}
    \osc_{Q_{\rho/4}^{(\omega/4)^{2-p}}}u\leq \left(1-\frac{\xi_0}{16}\right)\omega= \left(1-\frac{\sigma^{\frac{sp}{p-1}}\nu_d^{\frac1{p-2}}2^{-\frac{sp}{p-2}}}{16}\right)\omega
\end{align}
if 
\begin{equation}\label{condf.gamma}
    \gamma\coloneqq \left(2^{\frac{p-1}{sp}}+2^{\frac{p-1}{p-2}}\right)\left(8c_0\nu_d^{-\frac1{p-2}}\right)^{\frac{p-1}{sp}},
\end{equation}
which makes $\gamma$ still satisfy \eqref{cond1.gamma} and \eqref{cond2.gamma}.

\textbf{Second alternative : }
If \eqref{deg.ass2} does not hold, then for any $t\in(-r^{sp}(\omega/4)^{2-p},0]$,
\begin{equation*}
    \left|\left\{Q_{\rho}^{(\theta)}(0,t)\,:\,u-\mu_-\leq\omega/4\right\}\right|\geq \nu_0\overline{\omega}\left|Q_{\rho}^{(\theta)}\right|,
\end{equation*}
which gives
\begin{equation}\label{eq:mu+u}
    \left|\left\{Q_{\rho}^{(\theta)}(0,t)\,:\,\mu_+-u\geq\omega/4\right\}\right|\geq \nu_0\overline{\omega}\left|Q_{\rho}^{(\theta)}\right|,
\end{equation}
where the constants $\theta$ and $\overline{\omega}$ are determined in \eqref{choi.thevome}.
Then for each $t\in(-r^{sp}\theta,0]$, there is a time level $\overline{t}\in (t-\rho^{sp}\theta,t-\nu_0\overline{\omega}\rho^{sp}\theta/2]$ such that 
\begin{equation*}
    \left|\left\{B_{\rho}\,:\,\mu_+-u(\cdot,\overline{t})\geq\omega/4\right\}\right|\geq \nu_0\overline{\omega}|B_{\rho}|/2.
\end{equation*}
Otherwise
\begin{align*}
    \left|\left\{Q_{\rho}^{(\theta)}(0,t)\,:\,\mu_+-u\geq\omega/4\right\}\right|&\leq \int_{t-\rho^{sp}\theta}^{t-\nu_0\overline{\omega}\rho^{sp}\theta/2}\left|\left\{B_{\rho}\,:\,\mu_+-u(\cdot,\tau)\geq\omega/4\right\}\right|\,d\tau+\nu_0\overline{\omega}\rho^{sp}\theta|B_{\rho}|/2\\
    &< \nu_0\overline{\omega}|Q^{(\theta)}_{\rho}|,
\end{align*}
so it contradicts \eqref{eq:mu+u}. Therefore, we have 
\begin{align*}
    \left|\left\{B_{\rho}\,:\,\mu_+-u(\cdot,\overline{t})\geq\xi_3\omega/4\right\}\right|\geq\left|\left\{B_{\rho}\,:\,\mu_+-u(\cdot,\overline{t})\geq\omega/4\right\}\right|\geq \nu_0\overline{\omega}|B_{\rho}|/2,
\end{align*}
where $\xi_3\in(0,1]$ is determined later (see \eqref{cond.xi2}). 
Let us recall $\mu_{+}-\omega/4\geq \omega/4\geq\epsilon$, which follows from the fact that $\omega\equiv\omega_{i-1}$ in \eqref{eq:RQw} and \eqref{seq.omega}. Thus by Lemma \ref{lem.di3}, we get
\begin{align}\label{eq:Brho}
    \left|\left\{B_{\rho}\,:\,\mu_+-u(\cdot,\widetilde{\mathfrak{t}})\geq\tau\xi_3\omega/4\right\}\right|\geq \nu_0\overline{\omega}|B_{\rho}|/4\quad\text{ for any }\widetilde{\mathfrak{t}}\in(\overline{t},\overline{t}+(\overline{\omega}^{n+p+1}/c_1)(\xi_3\omega/4)^{2-p}\rho^{sp}],
\end{align}
where 
\begin{equation}\label{tau.det}
    \tau=\overline{\omega}/c_1
\end{equation}
with $c_1=c_1(n,s,p,\Lambda)\geq2$ and $\overline{\omega}$ in \eqref{choi.thevome}, whenever 
\begin{align}\label{deg.ass5}
   (\rho/R)^{\frac{sp}{p-1}}\mathrm{Tail}((u-\mu_+)_+;\mathcal{Q})= (\sigma/\gamma)^{\frac{sp}{p-1}}\mathrm{Tail}((u-\mu_+)_+;\mathcal{Q})\leq \xi_3\omega/4.
\end{align}
We note that from $\overline{t}\in(t-\rho^{sp}\theta,t-\nu_0\overline{\omega}\rho^{sp}\theta/2]$, there is a positive number $\xi_3\in(0,1]$ such that
\begin{equation}\label{cond.xi2}
    \xi_3\geq \left(\overline{\omega}^{n+p+1}/c_1\right)^{\frac1{p-2}}\coloneqq \xi_2
\end{equation}
and 
\begin{equation}\label{deg.time2}
    \overline{0}+(\overline{\omega}^{n+p+1}/c_1)(\xi_3\omega/4)^{2-p}\rho^{sp}=0,
\end{equation}
where $\overline{0}$ means the evaluation $t=0$ in $\overline{t}\in(t-\rho^{sp}\theta,t-\nu_0\overline{\omega}\rho^{sp}\theta/2]$. We point out that since $\overline{\omega}\leq1$ from \eqref{choi.thevome}, we ensure that $\xi_3\in(0,1]$.
Since we arbitrarily choose $t\in (-r^{sp}(\omega/4)^{2-p},0]$, \eqref{deg.time2} and \eqref{eq:Brho} implies 
\begin{align*}
    \left|\left\{B_{\rho}\,:\,\mu_+-u(\cdot,t)\geq\tau\xi_2\omega/4\right\}\right|\geq \nu_0\overline{\omega}|B_{\rho}|/4
\end{align*}
for any $t\in(-r^{sp}(\omega/4)^{2-p},0]$. For $\varsigma\in(0,1]$ which will be determined later in \eqref{cond1.varsigma}, $\tau$ in \eqref{tau.det}, $\xi_2$ in \eqref{cond.xi2}, we assume 
\begin{equation}\label{deg.sigma2}
    \sigma\leq (\varsigma\tau\xi_2)^{\frac{p-1}{sp}}
\end{equation}
to see that $(-\rho^{sp}(\varsigma\tau\xi_2\omega/4)^{2-p},0]\subset(-r^{sp}(\omega/4)^{2-p},0]$. We now use Lemma \ref{lem.di4} with $\alpha$ and $\xi$ replaced by $\nu_0\overline{\omega}/4$ and $\tau\xi_2/4$, respectively, to see that 
\begin{equation*}
    \left|\left\{Q_{\rho}^{(\varsigma\tau\xi_2\omega/4)^{2-p}}\,:\, \mu_{+}-u\leq \varsigma\tau\xi_2\omega/4\right\}\right|\leq \frac{c_2\varsigma^{p-1}}{\overline{\omega}}|Q_{\rho}^{(\varsigma\tau\xi_2\omega/4)^{2-p}}|
\end{equation*}
for some constant $c_2=c_2(n,s,p)$, whenever
\begin{equation}\label{deg.ass7}
    (\rho/R)^{\frac{sp}{p-1}}\mathrm{Tail}((u-\mu_+)_+;\mathcal{Q})=(\sigma/\gamma)^{\frac{sp}{p-1}}\mathrm{Tail}((u-\mu_+)_+;\mathcal{Q})\leq \varsigma\tau\xi_2\omega/4.
\end{equation}
We now choose 
\begin{equation}\label{cond1.varsigma}
\varsigma\coloneqq \left({\overline{\omega}}\nu_0c_2^{-1}\right)^{\frac1{p-1}}\leq\frac{1}{2}
\end{equation}
to see that
\begin{equation*}
    \left|\left\{Q_{\rho}^{(\varsigma\tau\xi_2\omega/4)^{2-p}}\,:\, \mu_{+}-u\leq \varsigma\tau\xi_2\omega/4\right\}\right|\leq \nu_0|Q_{\rho}^{(\varsigma\tau\xi_2\omega/4)^{2-p}}|.
\end{equation*}
By Remark \ref{rmk.nondeg} with $\xi=\varsigma\tau\xi_2/4$, we get 
\begin{align*}
    \mu_+-u\geq \varsigma\tau\xi_2\omega/8\quad\text{in }Q^{(\varsigma\tau\xi_2\omega/4)^{2-p}}_{\rho/2},
\end{align*}
whenever 
\begin{equation}\label{deg.ass8}
     (\rho/R)^{\frac{sp}{p-1}}\mathrm{Tail}((u-\mu_+)_+;\mathcal{Q})=(\sigma/\gamma)^{\frac{sp}{p-1}}\mathrm{Tail}((u-\mu_+)_+;\mathcal{Q})\leq \varsigma\tau\xi_2\omega/4.
\end{equation}
By \eqref{deg.sigma2} and \eqref{condf.gamma}, we have
\eqref{deg.ass5}, \eqref{deg.ass7} and \eqref{deg.ass8} hold. Therefore, we have 
\begin{align}\label{deg.res.sec}
    \osc_{Q_{\rho/2}^{(\varsigma\xi_2\omega/4)^{2-p}}}u\leq  \omega(1-\varsigma\tau\xi_2/8),
\end{align}
if 
\begin{equation}\label{condf.sigma}
    \sigma\coloneqq(\varsigma\tau\xi_2)^{\frac{p-1}{sp}}.
\end{equation}
By \eqref{deg.res.fir} and \eqref{deg.res.sec} together with \eqref{condf.gamma} and \eqref{condf.sigma}, we have 
\begin{align*}
    \osc_{Q_{\rho/4}^{(\omega/4)^{2-p}}}u\leq \left(1-\frac{\varsigma\tau\xi_2\nu_d^{\frac1{p-2}}2^{-\frac{sp}{p-2}}}{16}\right)\omega.
\end{align*}
In addition, in light of \eqref{choi.thevome} along with the fact that $\max\{\omega/4,1\}\leq \omega_0$, \eqref{tau.det}, \eqref{cond.xi2} and \eqref{cond1.varsigma}, we obtain
\begin{align*}
    \osc_{Q_{\rho/4}^{(\omega/4)^{2-p}}}u\leq \left(1-\frac{{\omega}^M}{c\omega_0^M}\right)\omega
\end{align*}
for some constant $M=M(n,s,p,\Lambda)\geq4$ and $c=c(n,s,p,\Lambda)\geq 4$. Therefore, by recalling the definition of the function $f_2$ given in \eqref{defn.f}, we get
\begin{equation*}
      \osc_{Q_{\rho/4}^{(\omega/4)^{2-p}}}u\leq f_2(\omega)\omega
\end{equation*}
by taking $M_2=M$ and $N_2=4\max\{c,2^{s}/(2^s-1)\}$ which depend only on $n,s,p$ and $\Lambda$.  In addition, by the choice of $N_2$, we get 
\begin{align}\label{cond.f2}
    f_2(\omega)\omega\geq \omega/2^s.
\end{align}
In addition, by \eqref{condf.sigma}, \eqref{cond.xi2}, \eqref{tau.det}, \eqref{cond1.varsigma}, and \eqref{condf.gamma}, there is a suitable constant $C_1\geq 1$ depending only on $n,s,p$ and $\Lambda$ such that $\sigma/\gamma={\overline{\omega}}^{M}/C_1$. Therefore, by taking 
\begin{align}\label{eq:N1M1}
N_1=2^{\frac{p-2}{p}+2+2M}{C_1}(\geq 4)\quad\text{and}\quad M_1=\frac{M(n+sp)}{sp}(\geq 4),
\end{align}
together with $\omega_0\geq \max\{1,\omega/4\}$ we have
\begin{equation}\label{eq:f_1}
    f_1(\omega)=\frac{\omega^{M_1}}{2^{\frac{p-2}{p}+2+2M_1}{C_1}\omega_0^{M_1}}\leq \frac{\sigma}{2^{\frac{p-2}{p}+2}\gamma}.
\end{equation}
Indeed, since $\frac{\omega}{4\omega_0}\leq\frac{\omega}{\max\{4,\omega\}}$, we obtain
\begin{align*}
\frac{\omega^{M_1}}{2^{2M_1}{C_1}\omega_0^{M_1}}\leq \dfrac{1}{C_1}\left(\dfrac{\omega}{\max\{2^2,\omega\}}\right)^{\frac{n+sp}{sp}M}=\dfrac{\overline{\omega}^M}{C_1}=\frac{\sigma}{\gamma}.
\end{align*}
Then with \eqref{eq:f_1}, by using $\omega_j\leq\omega_{j-1}=\omega$ and \eqref{seq.omega}, together with $N_1\geq 4$ and $\omega_{j-1}\leq\omega_0$, we get
\begin{equation*}
    \osc_{Q^{(\omega_j/4)^{2-p}}_{f_1(\omega_{j-1})\rho_{j-1}}}u\leq \osc_{Q_{\rho/4}^{(\omega/4)^{2-p}}}u\leq f_2(\omega)\omega=f_2(\omega_{j-1})\omega_{j-1}\leq\omega_j\quad\quad\text{and}\quad\quad Q_{j}\subset Q_{j-1}.
\end{equation*}
Indeed, note that $Q_{j}=B_j\times I_j=B_{\rho_j}\times I_{\rho_j^{sp}}^{(\omega_j/4)^{2-p}}$. Since $\omega_{j}\leq \omega_{j-1}$ and \eqref{defn.f}, it suffices to show $\rho_{j}^{sp}(\omega_j/4)^{2-p}\leq \rho_{j-1}^{sp}(\omega_{j-1}/4)^{2-p}$. 
Note from \eqref{defn.f} with $N_1\geq 4$ and \eqref{seq.omega} together with the fact that $\omega_{j}\geq \omega_{j-1}/2^s$,
\begin{align}\label{eq:rho.j}
\begin{split}
\rho_{j}^{sp}(\omega_j/4)^{2-p}=(f_1(\omega_{j-1})\rho_{j-1})^{sp}(\omega_j/4)^{2-p}&\leq (f_1(\omega_{j-1})\rho_{j-1})^{sp}(\omega_{j-1}/4)^{2-p}2^{s(p-2)}\\
&\leq\rho_{j-1}^{sp}(\omega_{j-1}/4)^{2-p}
\end{split}
\end{align}
holds. Therefore, we have proved \eqref{ind.int}.

We are now ready to prove Theorem \ref{thm.intlast}. 
\begin{proof}[Proof of Theorem \ref{thm.intlast}.]
Using \eqref{ind.int}, we first inductively prove that there is a constant $\varsigma=\varsigma(n,s,p,\Lambda)$ such that
\begin{align}\label{ineq3.uni}
    \osc_{Q_{i}}u\leq  \omega_{0}(1+i)^{-\varsigma}+4\epsilon
\end{align}
for some constant $c=c(n,s,p,\Lambda)$, where the cylinder $Q_i$ is determined in \eqref{seq.cylinder}.
By Lemma \ref{lem.iter}, there is a constant $\varsigma=\varsigma(n,s,p,\Lambda)\in(0,1)$ such that the sequence $a_i\coloneqq \omega_0(1+i)^{-\varsigma}$ satisfies
\begin{equation}\label{ineq1.uni}
    a_i\geq a_{i-1}f_2\left(a_{i-1}\right)\quad\text{for any }i\geq1.
\end{equation}
We now note that there is a positive integer $i_0$ such that $\omega_{i_0}=4\epsilon$. If not, we have $\omega_{i+1}=\omega_if_2(\omega_i)$ for any $i\geq 1$, as \eqref{seq.omega} and \eqref{cond.f2}. Then, by \eqref{ineq1.uni}, we have $\omega_{i}\leq \omega_0(1+i)^{-\varsigma}$. Therefore, if $i$ is sufficiently large, then $\omega_{i+1}\leq 4\epsilon$, which gives a contradiction.

Let us take the smallest positive integer $i_0$ such that $\omega_{i_0}=4\epsilon$. Then, we have $\osc_{Q_i}u\leq \omega_0(1+i)^{-\varsigma}$ for any $i\leq i_0$ and $\osc_{Q_i}u\leq 4\epsilon$ for $i>i_0$. This implies \eqref{ineq3.uni}.

Using \eqref{ineq3.uni}, we are now ready to prove \eqref{ineq.intlast}. Let take $r\in(0,\rho_0]$. Then there is a nonnegative integer $j$ such that 
\begin{equation*}
    \rho_{j+1}<r\leq \rho_j.
\end{equation*}
We next observe from \eqref{seq.omega} that $\omega_{j}\geq \omega_0/2^{sj}$ and
\begin{equation}\label{ineq2.cont}
    \rho_{j+1}\geq \left(\frac{\omega_j}{\omega_0}\right)^{M_1}\frac{1}{N_1}\rho_j\geq \cdots\geq \left(\frac{1}{N_1}\right)^{j+1}\left[\prod_{i=0}^{j}\left(\frac{1}{2^{si}}\right)^{M_1}\right]\rho_0,
\end{equation}
where the constants $N_1$ and $M_1$ fixed in \eqref{eq:N1M1} depend only on $n,s,p$ and $\Lambda$.
Using \eqref{ineq2.cont}, we get
\begin{equation*}
    \frac{\rho_0}{(2N_1)^{{2M_1j^2}}}\leq \frac{\rho_0}{N_1^{j+1}2^{\frac{M_1j(j+1)}{2}}}\leq \rho_{j+1}<r.
\end{equation*}
By taking logarithm on both sides together with a few computations, we have $j^2\geq c\ln(\rho_0/r)$, where $c=c(n,s,p,\Lambda)$. Using this along with \eqref{ineq3.uni}, we obtain
\begin{equation*}
    \osc_{Q_r^{(\omega_0/4)^{2-p}}}u\leq \osc_{Q_{{j}}}u\leq \omega_0(1+j)^{-\varsigma}+4\epsilon\leq c\omega_0\left(1+\ln(\rho_0/r)\right)^{-\varsigma/2}+4\epsilon
\end{equation*}
for some constant $c=c(n,s,p,\Lambda)$, which gives \eqref{ineq.intlast}. This completes the proof.
\end{proof}

\section{Estimates on the Lateral Boundary}\label{sec:4}
Let us fix $\epsilon\in(0,1)$. In this section, we are going to study a bounded weak solution
\begin{equation*}
    u\equiv u_\epsilon\in L^p(0,T;W^{s,p}(\Omega'))\cap C(0,T;L^2(\Omega))\cap L^\infty(0,T;L^{p-1}_{sp}(\bbR^n))
\end{equation*}to 
    \begin{equation}\label{eq.approxb}
\left\{
\begin{alignedat}{3}
\partial_t (u+\beta_\epsilon(u))+\mathcal{L}u&= 0&&\qquad \mbox{in  $\mathcal{Q}\cap \Omega_T$}, \\
u&=g&&\qquad  \mbox{in $\mathcal{Q}\setminus \Omega_T$},
\end{alignedat} \right.
\end{equation}
where $\Omega'\Supset\Omega$ and $\mathcal{Q}\coloneqq B_{\mathcal{R}}(x_0)\times (t_1-{T}_0,t_1]$ with $x_0\in \partial\Omega$, $\mathcal{R}\in(0,1]$, $t_1\in\setR$, $T_0>0$ and $\mathcal{Q}\subset\Omega'_T$. 
Let us fix constants 
\begin{align*}
    \mu_+\geq \sup_{\mathcal{Q}}u,\,\quad\mu_{-}\leq \inf_{\mathcal{Q}}u,\quad\text{and}\quad\omega\geq\mu_+-\mu_-.
\end{align*}
We next assume that the complement of $\Omega$ satisfies the measure density condition with $\alpha_0\in(0,1)$ (see \eqref{ass.density} for the precise definition of this). 
We assume $g\in C\left(\overline{\mathcal{Q}}\right)$ and there is a non-decreasing function $\omega_g:\bbR_+\to\bbR_+$ with $\omega_g(0)=0$ and
\begin{equation*}
    |g(z_1)-g(z_2)|\leq \omega_g\left(|x_1-x_2|+|t_1-t_2|^{\frac1{sp}}\right)\quad\text{for any }z_1,z_2\in\mathcal{Q}.
\end{equation*}
For $z_0\in\setR^{n+1}$ and $R,\mathcal{T}>0$, we fix 
\begin{equation*}
    Q_{R,\mathcal{T}}(z_0)\subset\mathcal{Q}.
\end{equation*}

We now observe the following energy estimates near the boundary.
\begin{lemma}\label{lem.energy2}
   Let $u$ be a sub(super)-solution to \eqref{eq.approxb}. Let $\phi$ be a positive smooth function satisfying $\phi(\cdot,t)\equiv0$ on $\bbR^n\setminus B_r(x_0)$ with $r<R$. Then there is a constant $c=c(n,s,p,\Lambda)$ such that
    \begin{equation}\label{ineq.energy2}
    \begin{aligned}
        &\sup_{t_0-\mathcal{T}<t<t_0}\left[\int_{B_R(x_0)}\phi^p(u-k)_{\pm}^{2}\,dx\pm\int_{B_R(x_0)}\phi^p\int_{k}^{u}\beta_\epsilon'(\xi)(\xi-k)_{\pm}\,d\xi\,dx\right]\\
        &\quad+ \int_{t_0-\mathcal{T}}^{t_0}\int_{B_R(x_0)}\int_{B_R(x_0)}\frac{|((u-k)_{\pm}\phi)(x,t)-((u-k)_{\pm}\phi)(y,t)|^p}{|x-y|^{n+sp}}\,dx\,dy\,dt\\
        &\quad+ \int_{t_0-\mathcal{T}}^{t_0}\int_{B_R(x_0)}\int_{B_R(x_0)}\frac{(u-k)_{\mp}^{p-1}(y,t)((u-k)_{\pm}\phi^p)(x,t)}{|x-y|^{n+sp}}\,dx\,dy\,dt\\
        &\leq cR^{p-sp}\int_{Q_{R,\mathcal{T}}(z_0)}(u-k)_{\pm}^p|\nabla\phi|^p\,dz
    + c\int_{Q_{R,\mathcal{T}}(z_0)}\left((u-k)_{\pm}^2\pm\int_{k}^{u}\beta_\epsilon'(\xi)(\xi-k)_{\pm}\,d\xi\right)\partial_t \phi^p\,dz\\
        &\quad+c\int_{t_0-\mathcal{T}}^{t_0}\int_{\bbR^n\setminus B_R(x_0)}\int_{B_R(x_0)}\frac{(u-k)_{\pm}^{p-1}(y,t)((u-k)_{\pm}\phi^p)(x,t)}{|x-y|^{n+sp}}\,dx\,dy\,dt\\
        &\quad +\int_{B_R(x_0)\times \{t=t_0-\mathcal{T}\}}\left((u-k)_{\pm}^2\pm\int_{k}^{u}\beta_\epsilon'(\xi)(\xi-k)_{\pm}\,d\xi\right)\phi^p\,dx,
    \end{aligned}
    \end{equation}
    whenever
\begin{equation}\label{k.energy2}
\begin{aligned}
    \begin{cases}
        k\geq \sup\limits_{\mathcal{Q} }g&\quad\mbox{if $u$ is a sub-solution}\\
        k\leq \inf\limits_{\mathcal{Q}}g&\quad\mbox{if $u$ is a super-solution}.
    \end{cases}
\end{aligned}
\end{equation}
\end{lemma}
\begin{proof}
    By \eqref{k.energy2}, we are able to test $(u-k)_{\pm}\phi^p$ to \eqref{eq.approxb}. Therefore, as in the proof of Lemma \ref{lem.energy}, we have the desired energy inequalities near the boundary.
\end{proof}
First, we prove a measure shrinking lemma near the boundary. In particular, we use a different kind of intrinsic cylinder which is of the form $Q_\rho^{(\overline{\theta})}$, where $\overline{\theta}\eqsim \omega^{1-p}$. Since we employ the third term given in the left-hand side of \eqref{ineq.energy2} which specially appears in nonlocal problems, such a cylinder is useful for dealing with the singular term $\beta_\epsilon$.

\begin{lemma}\label{lem.di4b}
    Let $u$ be a sub(super)-solution to \eqref{eq.approxb} with the measure density condition \eqref{ass.density} for $\Omega$. Let us fix $\xi,\varsigma\in(0,1/4]$ and write
    \begin{equation}\label{defn.thetab1}
    \begin{aligned}
        \overline{\theta}\coloneqq\
            (\varsigma\xi\omega)^{1-p}.
    \end{aligned}
    \end{equation} 
    If 
    \begin{equation}\label{ass.dib4}
        (\rho/R)^{\frac{sp}{p-1}}\mathrm{Tail}((u-\mu_\pm)_\pm;\mathcal{Q})\leq \varsigma\xi\omega\quad\text{and}\quad\xi\omega\leq \pm\mu_{\pm}-\sup_{\mathcal{Q}}(\pm g)
    \end{equation}
    with $Q_{\rho}^{(\overline{\theta})}(z_0)\subset \mathcal{Q}$, then 
    \begin{equation*}
        \left|\{\pm(\mu_{\pm}-u)\leq\varsigma\xi\omega\}\cap Q^{(\overline{\theta})}_\rho(z_0)\right|\leq {c\varsigma^{{p-1}}}\max\{1,\varsigma\xi\omega\}\left|Q^{(\overline{\theta})}_\rho\right|
    \end{equation*}
    for some constant  $c=c(n,s,p,\Lambda,\alpha_0)$, where the constant $\alpha_0$ is determined in \eqref{ass.density}. 
\end{lemma}
\begin{proof}
    It suffices to consider the case when $u$ is a super-solution and $z_0=0$. By the maximum principle, \eqref{ass.density} and the second condition in \eqref{ass.dib4}, we observe 
    \begin{align*}
        \left|\{-(\mu_{-}-u)(\cdot,t)\geq\xi\omega\}\cap B_\rho\right|\geq \alpha_0|B_{\rho}|
    \end{align*}
    for any $t\in I_{\rho}^{(\overline{\theta})}$. Therefore, together with $0\leq \phi=\phi(x)\in C_c^\infty(B_{3\rho/2})$, $\phi\equiv 1$ on $B_\rho$, and $\abs{\nabla\phi}\leq c/\rho$, we have
    \begin{equation}\label{ineq1.dib4}
    \begin{aligned}
        &\int_{Q_{\rho}^{(\overline{\theta})}}\int_{B_\rho}\frac{(u-k)_{+}^{p-1}(y,t)((u-k)_{-}\phi^p)(x,t)}{|x-y|^{n+sp}}\,dx\,dy\,dt\\
        &\geq \frac{(\xi\omega)^{p-1}}{c}\left(\int_{Q_{\rho}^{(\overline{\theta})}}\frac{((u-k)_{-}\phi^p)(x,t)}{|x-y|^{n+sp}}\,dz\right)|B_\rho|\\
        &\geq  \frac{\varsigma(\xi\omega)^p}{c\rho^{sp}}\left|\{-(\mu_{-}-u)\leq\varsigma\xi\omega\}\cap Q^{(\overline{\theta})}_\rho\right|,
    \end{aligned}
    \end{equation}
    where we write $k\coloneqq \mu_-+2\varsigma\xi\omega$. In addition, as in the estimate of \eqref{ineq03.di1} together with the first condition in \eqref{ass.dib4}, we deduce 
    \begin{align*}
      J\coloneqq &\rho^{-sp}\int_{Q_{2\rho}^{(\overline{\theta})}}(u-k)_-^p\,dz+\int_{Q_{2\rho}^{(\overline{\theta})}}\int_{\bbR^n\setminus B_{2\rho}}\frac{(u-k)_{-}^{p-1}(y,t)((u-k)_-\phi^p)(x,t)}{|x-y|^{n+sp}}\,dx\,dy\,dt\\
        &\quad +\int_{B_{2\rho}\times \{t=-(2\rho)^{sp}\overline{\theta}\}}\left((u-k)_{-}^2-\int_{u}^{k}\beta_\epsilon'(\xi)(\xi-k)_{-}\,d\xi\right)\phi^p\,dx \\
        &\leq c\overline{\theta}(\varsigma\xi\omega)^p|B_{\rho}|
        +c\left((\varsigma\xi\omega)^2+(\varsigma\xi\omega)\right)|B_\rho|\\
        &\leq \frac{c}{\rho^{sp}}\left[(\varsigma\xi\omega)^p+\overline{\theta}^{-1}\left((\varsigma\xi\omega)^2+(\varsigma\xi\omega)\right)\right]\left|Q^{(\overline{\theta})}_\rho\right|
    \end{align*}
    for some constant $c=c(n,s,p)$, where \eqref{ass.dib4} is used.
    We now use \eqref{defn.thetab1} to further estimate $J$ as 
    \begin{align*}
        J\leq \frac{c}{\rho^{sp}}(\varsigma\xi\omega)^p\max\{1,\varsigma\xi\omega\}\left|Q^{(\overline{\theta})}_\rho\right|.
    \end{align*}
    Using this, \eqref{ineq1.dib4} and the energy estimates given in Lemma \ref{lem.energy2} for $\phi=\phi(x)$ (so $\partial_t\phi\equiv 0$) with $r=\rho$ and $R=2\rho$, we have 
    \begin{align*}
        \frac{\varsigma(\xi\omega)^p}{\rho^{sp}}\left|\{-(\mu_{-}-u)\leq\varsigma\xi\omega\}\cap Q^{(\overline{\theta})}_\rho\right|\leq\frac{c}{\rho^{sp}}(\varsigma\xi\omega)^p\max\{1,\varsigma\xi\omega\}\left|Q^{(\overline{\theta})}_\rho\right|
    \end{align*}
    for some constant $c=c(n,s,p,\Lambda)$, which completes the proof.
\end{proof}

We next observe a De Giorgi type lemma near the boundary.

\begin{lemma}\label{lem.di1b}
    Let $u$ be a sub(super)-solution to \eqref{eq.approxb} with the measure density condition \eqref{ass.density} for $\Omega$. Let us fix $\xi\in(0,1/4]$. We take 
    \begin{equation}\label{defn.thetab}
    \begin{aligned}
        \overline{\theta}\coloneqq
            (\xi\omega)^{1-p}.
    \end{aligned}
    \end{equation}Suppose 
    \begin{align*}
        (\rho/R)^{\frac{sp}{p-1}}\mathrm{Tail}((u-\mu_{\pm})_{\pm};\mathcal{Q})\leq \xi\omega
        \quad \text{and}\quad\xi\omega\leq \pm\mu_{\pm}-\sup_{\mathcal{Q}}(\pm g)
    \end{align*}
    and $Q_\rho^{(\overline{\theta})}(z_0)\subset\mathcal{Q}$. Then there is a constant $\nu_1=\nu_1(n,s,p,\Lambda)\in(0,1)$ such that if 
    \begin{align*}
        |\{ Q_{\rho}^{(\overline{\theta})}(z_0)\,:\pm(\mu_{\pm}-u)\leq \xi\omega\}|\leq \nu_1\frac{\xi\omega}{\max\{1,\xi\omega\}^{1+n/sp}}|Q_\rho^{(\overline{\theta})}(z_0)|,
    \end{align*}
    then 
    \begin{equation*}
       \pm(\mu_{\pm}-u)\geq\xi\omega/2\quad\text{in  }Q_{\rho/2}^{(\overline{\theta})}(z_0).
    \end{equation*}
\end{lemma}
\begin{proof}
We may assume $z_0=0$ by Lemma \ref{lem.trans}. It suffices to consider the case when $u$ is a super-solution.
    Let us choose sequences $k_i,\rho_i$ and $\overline{\rho}_i$ as in \eqref{seq.di1}.
    We next define 
    \begin{equation*}
        A_i\coloneqq\frac{\left|\left\{Q_{\rho_i}^{(\overline{\theta})}\,:\,u<k_i\right\}\right|}{\left|Q_{\rho_i}^{(\overline{\theta})}\right|}
    \end{equation*}
    to see that 
    \begin{equation}\label{goal.di1b}
        A_{i+1}\leq cM^{i}\frac{\max\{1,\xi\omega\}^{1+sp/n}}{(\xi\omega)^{\frac{sp}n}} A_i^{1+sp/n}
    \end{equation}
    for some constant $c=c(n,s,p,\Lambda)$ and $M=M(n,s,p)$. To do this, we first recall the inequality \eqref{ineq01.di1} and we will estimate the terms $J_1$ and $J_2$ given in \eqref{ineq01.di1} with $\theta$ replaced by $\overline{\theta}$.
    Let us choose functions $\psi\in C_c^{\infty}(B_{\overline{\rho}_i})$ with $0\leq\psi\leq 1$, $\psi\equiv1$ on $B_{\rho_{i+1}}$ and $\oldphi\in C^\infty(\bbR)$ with $0\leq\oldphi\leq 1$, $\oldphi\equiv1$ on $t\geq -\rho_{i+1}^{sp}\overline{\theta}$ and $\oldphi\equiv0$ on $t<-\overline{\rho}_{i}^{sp}\overline{\theta}$.
    By using Lemma \ref{lem.energy2} with $k={k}_i$ and $\phi=\psi\oldphi$  and following the same lines as in the estimates of \eqref{ineq1.di1} and \eqref{ineq2.di1} with $\theta$ replaced by $\overline{\theta}$, we deduce
    \begin{align}\label{ineq1.di1b}
        J_1&\leq c\left[\dashint_{Q_{\rho_i}^{(\overline{\theta})}}2^{i(n+sp+p)}(u-{k}_{i})_-^p+
    \frac{2^{i}}{\overline{\theta}}\left((u-{k}_{i})_-^2+(1+\xi\omega)(u-{k}_{i})_-\right)\,dz\right]
    \end{align}
    and
    \begin{align}\label{ineq2.di1b}
        (\rho_{i+1}^{sp}\overline{\theta})^{-1}J_2^{\frac{n}{sp}}\leq c\rho_i^{-sp}\left[\dashint_{Q_{\rho_i}^{(\overline{\theta})}}2^{i(n+sp+p)}(u-{k}_{i})_-^p+\frac{2^{i}}{\overline{\theta}}\left((u-{k}_{i})_-^2+(1+\xi\omega)(u-{k}_{i})_-\right)\,dz\right]
    \end{align}
    for some constant $c=c(n,s,p,\Lambda)$.
    By the choice of the constant $\overline{\theta}$ determined in \eqref{defn.thetab}, we have 
    \begin{align*}
        \dashint_{Q_{\rho_i}^{(\overline{\theta})}}\overline{\theta}^{-1}\left((u-{k}_{i})_-^2+(1+\xi\omega)(u-{k}_{i})_-\right)\,dz\leq (\xi\omega)^{p-1}\left((\xi\omega)^{2}+(\xi\omega)\right)A_i.
    \end{align*}
    Combining this, \eqref{ineq1.di1b} and \eqref{ineq2.di1b} yields $J_1\leq c2^{i(n+sp+p)}\max\{1,\xi\omega\}(\xi\omega)^{p}A_i$ and
    \begin{align*}
        J_2\leq c\left(\overline{\theta}2^{i(n+sp+p)}\max\{1,\xi\omega\}(\xi\omega)^{p}A_i\right)^{\frac{sp}n}\leq c\left(2^{i(n+sp+p)}\max\{1,\xi\omega\}(\xi\omega)A_i\right)^{\frac{sp}n}
    \end{align*}
    for some constant $c=c(n,s,p,\Lambda)$, which implies 
    \begin{align*}
        \dashint_{Q^{(\overline{\theta})}_{\rho_{i+1}}}{(u-{k}_{i})^{p(1+2s/n)}_-}\,dz\leq c2^{i(n+sp+p)(1+sp/n)}\max\{1,\xi\omega\}^{1+sp/n}(\xi\omega)^{p(1+s/n)}A_i^{1+sp/n},
    \end{align*}
    where we have used \eqref{ineq01.di1}.
    Therefore, we have 
    \begin{align*}
        A_{i+1}&\leq \frac{c2^{ip(1+2s/n)}}{(\xi\omega)^{p(1+2s/n)}}\dashint_{Q_{\rho_{i+1}}^{(\overline{\theta})}}{(u-{k}_{i})^{p(1+2s/n)}_-}\,dz\\
        &\leq {c2^{ip(n+sp+p)(1+sp/n)}}\frac{\max\{1,\xi\omega\}^{1+sp/n}}{(\xi\omega)^{\frac{sp}n}}A_i^{1+sp/n}
    \end{align*}
    for some constant $c=c(n,s,p,\Lambda)$, which proves \eqref{goal.di1b}. Therefore, by Lemma \ref{lem.tech1}, we derive the desired result.
\end{proof}

\subsection{Uniform continuity}
In this subsection, we are going to prove Theorem \ref{thm.bdylast}. As in Subsection \ref{sec3.1}, we first construct a sequence of cylinders $Q_i$ and derive oscillation estimates on such cylinders (see \eqref{seq.cylinderb} and \eqref{ind.boundary}).

Let us fix $Q_\mathcal{R}(z_0)$ with $x_0\in \partial\Omega$ and $Q_{\mathcal{R}}(z_0)\subset\Omega'_T$. We now set
\begin{equation*}
    \omega_0\coloneqq\max\{\|u\|_{L^\infty(Q_\mathcal{R}(z_0))}+\mathrm{Tail}(u;Q_\mathcal{R}(z_0)),1\}.
\end{equation*}
Let us choose $\rho_0<\mathcal{R}$ such that
\begin{equation*}
    Q_{\rho_0}^{(\omega_0/4)^{1-p}}(z_0)\subset Q_{\mathcal{R}}(z_0).
\end{equation*}
We define 
\begin{equation}\label{defn.f1b}
    f_1(x)\coloneqq \frac{|x|^{M_1}}{N_0N_1\omega_0^{L_1}}\quad\text{and}\quad f_2(x)=1-\frac{|x|^{M_2}}{N_2\omega_0^{L_2}},
\end{equation}
where $M_1,M_2,N_1,N_2,L_1,L_2\geq4$ with $M_i\leq L_i$ are the constants determined later. In addition, we may assume $N_0\geq 1$ and this will be determined later to obtain a suitable modulus of continuity of the solution (see \eqref{choi.cg} below).
Take sequences
\begin{equation}\label{seq.omegab}
    \rho_{i+1}\coloneqq {f_1}(\omega_{i})\rho_i\quad\text{and}\quad \omega_{i+1}\coloneqq\max\left\{\omega_if_2(\omega_i),\omega_i/2^s,2\osc_{Q_{i}}g,4\epsilon\right\},
\end{equation}
where 
\begin{equation}\label{seq.cylinderb}
    {Q_i}\coloneqq Q^{(\overline{\theta_i})}_{\rho_i},\quad \overline{\theta_i}\coloneqq(\omega_i/4)^{1-p},\quad\mu_{i}^+\coloneqq \sup_{Q_{i}}u\quad\text{and}\quad \mu_{i}^-\coloneqq \sup_{Q_{i}}u-\omega_i.
\end{equation}

We are now going to determine the constants $M_1,M_2,N_1,N_2,L_1$ and $L_2$ to obtain that
\begin{equation}\label{ind.boundary}
    \osc_{Q_i}u\leq \omega_i\quad\text{and}\quad Q_{i}\subset Q_{i-1} \quad\text{for each }i\geq0,
\end{equation}
where $Q_{-1}=\mathcal{Q}_{\mathcal{R}}(z_0)$.
Suppose that 
\begin{equation*}
    \osc_{Q_i}u\leq \omega_i\quad\text{and}\quad Q_{i}\subset Q_{i-1} 
\end{equation*}
hold for any $i=0,1,\ldots,j-1$.
By following the same lines as in the proof of \eqref{ineq.tail} together with the fact that $f_1(\omega_i)\leq 1/2$ (from $M_1,N_1\geq 4$ and $\omega_i\leq\omega_0$) and $\omega_{i+1}\geq \omega_i/2^s$, by \eqref{seq.omega} we have 
\begin{equation}\label{ineq.tail2}
\mathrm{Tail}((u-\mu_{i}^{\pm})_\pm;Q_{i})\leq c_1\omega_{i}\quad\text{for any}\,\,i=0,1,\dots,j-1
\end{equation} 
for some constant $c_1=c_1(n,s,p)$, which is independent of $i$.
For a convenience of notations, we write 
\begin{equation*}
    R\coloneqq \rho_{j-1},\quad \mathcal{Q}\coloneqq Q_{j-1},\quad\omega\coloneqq \omega_{j-1},\quad \mu_+\coloneqq \mu_{j-1}^{+}\quad\text{and}\quad \mu_-\coloneqq \mu_{j-1}^{-}.
\end{equation*}
In addition, by Lemma \ref{lem.trans}, we may assume $z_0=0$. We now fix 
\begin{align}\label{choi.rhob}
    \rho\coloneqq \sigma R,
\end{align}
where $\sigma\in(0,1)$ will be determined in \eqref{choi.sigmab}.

We first consider the following two cases:
\begin{align}\label{alt.bdd}
    \mu_{+}-\omega/4\geq \sup_{\mathcal{Q}}g\quad\text{and}\quad \mu_-+\omega/4\leq \inf_{\mathcal{Q}}g.
\end{align}
It implies that 
\begin{align*}
    \omega/4\leq \pm\mu_{\pm}-\sup_{\mathcal{Q}}(\pm g).
\end{align*}
Let $\varsigma\in(0,1/4]$ be a free parameter determined later in \eqref{choi.varsigmab}. By Lemma \ref{lem.di4b} with $\xi=1/4$, we have 
    \begin{align}\label{eq:mu.u}
        \left|\left\{\pm(\mu_{+}-u)\leq \varsigma\omega/4\right\}\cap Q_{\rho}^{(\varsigma\omega/4)^{1-p}}\right|\leq {c^*\varsigma^{p-1}}\max\{1,\varsigma\omega/4\}\left|Q_{\rho}^{(\varsigma\omega/4)^{1-p}}\right|
    \end{align}
    for some constant $c^*=c^*(n,s,p,\Lambda,\alpha_0)\geq 1$ whenever 
    \begin{equation}\label{tail.ineq33}
        Q_{\rho}^{(\varsigma\omega/4)^{1-p}}\subset\mathcal{Q}\quad\text{and}\quad(\rho/R)^{\frac{sp}{p-1}}\mathrm{Tail}((u-\mu_\pm)_\pm;\mathcal{Q})\leq \varsigma\omega/4,
    \end{equation} where the constant $\alpha_0\in(0,1)$ is determined in \eqref{ass.density}. With $\nu_1$ from Lemma \ref{lem.di1b}, let us take
    \begin{equation}\label{choi.varsigmab}
        \varsigma\coloneqq\left(\frac{\nu_1\omega}{16c^*\omega_0^{2+n/sp}}\right)^{\frac1{p-2}}
    \end{equation}
    to see that 
    \begin{align*}
    \left|\left\{\pm(\mu_{+}-u)\leq \varsigma\omega/4\right\}\cap Q_{\rho}^{(\varsigma\omega/4)^{1-p}}\right|\leq \frac{\nu_1(\varsigma\omega/4)}{\omega_0^{1+\frac{n}{sp}}}\left|Q_{\rho}^{(\varsigma\omega/4)^{1-p}}\right|\leq \frac{\nu_1\varsigma\omega/4}{\max\{1,\varsigma\omega/4\}^{1+\frac{n}{sp}}}\left|Q_{\rho}^{(\varsigma\omega/4)^{1-p}}\right|,
   \end{align*}
where the constant $\nu_1=\nu_1(n.s,p,\Lambda)\in(0,1)$ is determined in Lemma \ref{lem.di1b}. Indeed, for the second inequality we have used $\omega_0\geq\max\{1,\varsigma\omega/4\}$ since $\varsigma\leq 1$ and \eqref{seq.omegab} hold, and for the first inequality
\begin{align*}
{c^*\varsigma^{p-2}}\varsigma\max\{1,\varsigma\omega/4\}=\frac{\nu_1(\varsigma\omega/4)}{\omega_0^{1+n/sp}}\dfrac{\max\{1,\varsigma\omega/4\}}{4\omega_0}\leq \frac{\nu_1(\varsigma\omega/4)}{\omega_0^{1+n/sp}}
\end{align*}
holds and so \eqref{eq:mu.u} is applicable.

We now apply Lemma \ref{lem.di1b} with $\xi=\varsigma/4\leq\frac{1}{4}$ to get that 
    \begin{align*}
        \pm(\mu_\pm-u)\geq\varsigma\omega/8\quad\text{in }Q^{(\varsigma\omega/4)^{1-p}}_{\rho/2},
    \end{align*}
    if \begin{equation}\label{tail.ineq3}
        Q_{\rho}^{(\varsigma\omega/4)^{1-p}}\subset\mathcal{Q}\quad\text{and}\quad(\rho/R)^{\frac{sp}{p-1}}\mathrm{Tail}((u-\mu_\pm)_\pm;\mathcal{Q})\leq \varsigma\omega/4.
    \end{equation}
    We now take 
    \begin{align}\label{choi.sigmab}
    \sigma=\left(\frac{\varsigma}{4c_1}\right)^{\frac{p-1}{sp}},
    \end{align}
    to make \eqref{tail.ineq33} and \eqref{tail.ineq3} holds true by \eqref{choi.rhob}, where the constants $c_1$ and $\varsigma$ are determined in \eqref{ineq.tail2} and \eqref{choi.varsigmab}, respectively. Therefore, we get
    \begin{align*}
        \osc_{Q_{\rho/2}^{(\omega/4)^{1-p}}}u\leq \osc_{Q_{\rho/2}^{(\varsigma\omega/4)^{1-p}}}u\leq \omega(1-\varsigma/8).
    \end{align*}
    By recalling \eqref{defn.f1b} and \eqref{choi.varsigmab}, there are constants $M_2\coloneqq 4+\frac{1}{p-2}$, $L_2\coloneqq \left(2+\frac{n}{sp}\right)\left(4+\frac{1}{p-2}\right)> M_2$ and $N_2=N_2(n,s,p,\Lambda,\alpha_0)\geq \max\{2^s/(2^s-1)+M_2,(16c^*/\nu_1)^{\frac1{p-2}}\}$ such that \eqref{cond.f2} and
    \begin{align}\label{choi.n22}
        \osc_{Q_{\rho/2}^{(\varsigma\omega/4)^{1-p}}}u\leq \omega \left(1-\frac{\omega^{M_2}}{N_2\omega_0^{L_2}}\right)=\omega f_2(\omega).
    \end{align}
    In addition, by \eqref{choi.rhob}, \eqref{choi.varsigmab} and \eqref{choi.sigmab}, there are constants $M_1\coloneqq 4+M_2(p-1)/(sp)$, $L_1\coloneqq 4+L_2(p-1)/(sp)$ and $C_1=C_1(n,s,p,\Lambda,\alpha_0)\geq 1$ such that $\rho/2=\frac{\omega^{\frac{1}{p-2}\frac{p-1}{sp}}}{C_1 \omega_0^{(2+\frac{n}{sp})\frac1{p-2}\frac{p-1}{sp}}}R$ and $\frac{\omega^{M_1}}{C_1\omega_0^{L_1}}\leq\frac{1}{2}$.

We now choose
\begin{align}\label{choi.n1}
    f_1(x)=\frac{|x|^{M_1}}{N_0N_1\omega_0^{L_1}},
\end{align}
where $N_1\coloneqq (4^{\frac{2(p-1)}{s}}+1)C_1$. We point out that a constant $N_0\geq1$ will be chosen later depending only on a given data (see \eqref{choi.cg}). Then similar to the argument for \eqref{eq:rho.j}, we get
\begin{align*}
\osc_{Q_{f_1(\omega_{j-1})\rho_{j-1}}^{(\omega_j/4)^{1-p}}}u\leq\osc_{Q_{\rho/2}^{(\omega/4)^{1-p}}}u\leq f_2(\omega_{j-1})\omega_{j-1}\leq \omega_j\quad\quad\text{and}\quad\quad Q_{j}\subset Q_{j-1}.
\end{align*}
On the other hand, if the two cases in \eqref{alt.bdd} do not occur, then $\osc_{Q_{j}}u\leq \osc_{Q_{j-1}}u\leq 2\osc_{Q_j}g\leq \omega_{j}$ holds. Therefore, we have proved \eqref{ind.boundary}.

Using this, we are now able to derive a quantitative oscillation estimate \eqref{ineq.bdylast}. 
\begin{proof}[Proof of Theorem \ref{thm.bdylast}.]
By \eqref{ass.bdylast}, we take $\omega_0=1$ and take $\rho_0=\mathcal{R}/4^{\frac{p-1}{sp}}$ to see that $Q_{\rho_0}^{(\omega_0/4)^{1-p}}(z_0)\subset {Q}_{\mathcal{R}}(z_0)$.
By Lemma \ref{lem.iter}, there is a constant $\varsigma=\varsigma(n,s,p,\Lambda,\alpha_0)\in(0,1)$ such that the sequence $a_i\coloneqq (1+i)^{-\varsigma}$ satisfies 
\begin{equation}\label{ineq0000.unib}
    a_n\geq a_{n-1}f_2(a_{n-1}),
\end{equation} as the constants $M_2,N_2$ and $L_2$ depend only on $n,s,p,\Lambda$ and $\alpha_0$.

We now assume 
\begin{align}\label{omegag}
    \omega_g(\rho)\leq c_g\left(1+\ln(\mathcal{R}/\rho)\right)^{-\delta}
\end{align}
holds for any $\rho\in(0,\mathcal{R}]$, where $\delta\in(\varsigma,1)$ and $c_g\geq1$.
We are now going to find a parameter $N_0$ which is first introduced in \eqref{defn.f1b}, in order to derive 
\begin{align*}
    \osc_{Q_i}u\leq \left(1+i\right)^{-\varsigma}+4\epsilon
\end{align*}
for any $i\geq0$.

In light of \eqref{defn.f1b} and the fact that $\omega_{i}\leq\omega_0=1$, we get
\begin{align}\label{ineq000.unib}
    \rho_{i}\leq \frac{\rho_{i-1}}{N_0N_1}\leq \cdots\leq \frac{\rho_0}{(N_0N_1)^{i+1}}.
\end{align}
Using \eqref{cond.oscg}, $\omega_{i+1}\geq \omega_i/2^s$, \eqref{ineq000.unib}, the fact that $N_1\geq 2^{\frac{2(p-1)}{s}}$ (which is determined in \eqref{choi.n1}), and \eqref{omegag}, we have 
\begin{equation*}
\begin{aligned}
    \osc_{Q_i}g\leq \omega_g\left(\rho_i\left[1+(\omega_i/4)^{\frac{1-p}{sp}}\right]\right)&\leq \omega_g\left(\rho_i2^{i(p-1)/p}\left[1+(\omega_0/4)^{\frac{1-p}{sp}}\right]\right)\\
    &\leq \omega_g\left(\rho_0/N_0^{i+1}\right)\leq c_g\left(1+(i+1)\ln N_0\right)^{-\delta},
\end{aligned}
\end{equation*}
where we also have used the fact that $\omega_0=1$ and $\rho_0=\mathcal{R}/4^{\frac{p-1}{sp}}$.
We now choose $N_0=N_0(\delta,c_g)\geq1$ sufficiently large to see that 
\begin{equation}\label{choi.cg}
    \osc_{Q_i}g\leq (2+i)^{-\delta}/2,
\end{equation}
which implies $\osc_{Q_i}g\leq (2+i)^{-\varsigma}/2$, by the choice of $\delta>\varsigma$.

Therefore, by considering the proof of \eqref{ineq3.uni}, we similarly derive
\begin{align}\label{ineq00.unib}
    \osc_{Q_{i}}u\leq  (1+i)^{-\varsigma}+4\epsilon
\end{align}
for any $i\geq0$. First, by following the same lines as in the proof of \eqref{ineq3.uni}, there is the smallest positive integer $i_0$ such that $\omega_{i_0+1}=4\epsilon$ and $\omega_{i_0}>4\epsilon$. Then we inductively prove \begin{align}\label{ineq001.unib}
\omega_i\leq  (1+i)^{-\varsigma}
\end{align} for any $i\in[0,i_0]$. 
Suppose \eqref{ineq001.unib} holds for any $i=0,1,\ldots, j-1$. Then we have $\omega_{j-1}\leq j^{-\varsigma}$ and
\begin{equation*}
    \omega_{j}\leq\max\{\omega_{j-1}f_2(\omega_{j-1}),(1+j)^{-\varsigma}\}.
\end{equation*}
We next note that $xf_2(x)$ is non-decreasing function when $|x|\leq1$ by the fact that $N_2\geq M_2$ which is determined in \eqref{choi.n22}. Therefore, using this and \eqref{ineq0000.unib}, we have 
\begin{align*}
    (1+j)^{-\varsigma}\geq j^{-\varsigma}f_2(j^{-\varsigma})\geq \omega_{j-1}f_2(\omega_{j-1}),
\end{align*}
which implies \eqref{ineq001.unib} with $i=j$. Thus, we have proved \eqref{ineq00.unib}.

Using \eqref{ineq00.unib}, we are now ready to prove \eqref{ineq.bdylast}. We may assume $r\in(0,\rho_0]$, since if $r\in(\rho_0,\mathcal{R}]$, then \eqref{ineq.bdylast} directly follows by taking the constant $c=c(n,s,p,\delta)$ sufficiently large.
We then note that there is a nonnegative integer $i$ such that 
\begin{align*}
    \rho_{i+1}<r\leq\rho_i.
\end{align*}
From \eqref{seq.omegab} and the fact that $\omega_0=1$, we observe $\omega_{i}\geq \omega_0/2^{si}=1/2^{si}$ and
\begin{equation*}
   r> \rho_{i+1}\geq\frac{\omega_i^{M_1}}{N_0N_1}\rho_i\geq \frac{1}{2^{M_1si}}\frac{1}{N_0N_1}\rho_i\geq\cdots\geq \left(\frac{1}{N_0N_1}\right)^{i+1}\left[\prod_{k=0}^{i}\left(\frac{1}{2^{sk}}\right)^{M_1}\right]\rho_0,
\end{equation*}
where $N_1=N_1(n,s,p,\Lambda,\alpha_0)$, $M_1=M_1(n,s,p,\Lambda)$ are defined preceding \eqref{choi.n1}, and the constant $N_0=N_0(\delta,c_g)$ is determined in \eqref{choi.cg}.
Thus, we deduce
\begin{equation*}
    \frac{\rho_0}{\left(2N_0N_1\right)^{{2M_1i^2}}}\leq \frac{\rho_0}{\left(N_0N_1\right)^{i+1}2^{\frac{M_1i(i+1)}{2}}}\leq r,
\end{equation*}
which implies $i^2\geq c\,{\ln(\rho_0/r)}$, where $c=c(n,s,p,\Lambda,\alpha_0,\delta,c_g)$. Plugging this into \eqref{ineq00.unib} together with the fact that $\rho_0=\mathcal{R}/4^{\frac{p-1}{sp}}$ yields 
\begin{align*}
    \osc_{Q_r^{(\omega_0/4)^{2-p}}}u\leq \osc_{Q_i}u\leq (1+i)^{-\varsigma}+4\epsilon\leq c\left(1+\ln(\mathcal{R}/r)\right)^{-\varsigma/2}+4\epsilon.
\end{align*}
Since we assume $\omega_0=1$, after a few computations, we derive \eqref{ineq.bdylast}. This completes the proof.
\end{proof}

\section{Estimates on the Initial Boundary}\label{sec:5}
In this section, we are going to study a weak solution $u\equiv u_\epsilon$ to 
    \begin{equation}\label{eq.approxi}
\left\{
\begin{alignedat}{3}
\partial_t (u+\beta_\epsilon(u))+\mathcal{L}u&= 0&&\qquad \mbox{in  $\mathcal{Q}\cap \Omega_T$}, \\
u&=g&&\qquad  \mbox{in $\mathcal{Q}\setminus \Omega_T$,}
\end{alignedat} \right.
\end{equation}
where $\Omega'\Supset\Omega$ and $\mathcal{Q}\coloneqq B_R(x_0)\times [0,T_0]$ with $x_0\in \overline{\Omega}$, and $B_R(x_0)\times (0,T_0]\subset\Omega'_T$. Let
\begin{align*}
    \mu_+\geq \sup_{\mathcal{Q}}u,\,\quad\mu_{-}\leq \inf_{\mathcal{Q}}u,\quad\text{and}\quad\omega\geq\mu_+-\mu_-.
\end{align*}
We assume $g\in C\left(\overline{\mathcal{Q}}\right)$ and there is a non-decreasing function $\omega_g:\bbR_+\to\bbR+$ with $\omega_g(0)=0$ and
\begin{equation*}
    |g(z_1)-g(z_2)|\leq \omega_g\left(|x_1-x_2|+|t_1-t_2|^{\frac1{sp}}\right)\quad\text{for any }z_1,z_2\in\mathcal{Q}.
\end{equation*}

Let us fix $(0,\mathcal{T})\subset (0,T_0)$. Then we have the following energy estimates near the initial level.
\begin{lemma}\label{lem.energy3}
   Let $u$ be a sub(super)-solution to \eqref{eq.approxi}. Let $\phi=\phi(x)$ be a cutoff function satisfying $\phi\equiv0$ on $\bbR^n\setminus B_r(x_0)$ with $r<R$. Then there is a constant $c=c(n,s,p,\Lambda)$ such that
    \begin{align*}
        &\sup_{0<t<\mathcal{T}}\int_{B_R(x_0)}\phi^p(u-k)_{\pm}^{2}\,dx\\
        &\quad+ \int_{0}^{\mathcal{T}}\int_{B_R(x_0)}\int_{B_R(x_0)}\frac{|((u-k)_{\pm}\phi)(x,t)-((u-k)_{\pm}\phi)(y,t)|^p}{|x-y|^{n+sp}}\,dx\,dy\,dt\\
        &\quad+ \int_{0}^{\mathcal{T}}\int_{B_R(x_0)}\int_{B_R(x_0)}\frac{(u-k)_{\mp}^{p-1}(y,t)((u-k)_{\pm}\phi^p)(x,t)}{|x-y|^{n+sp}}\,dx\,dy\,dt\\
        &\leq cR^{p-sp}\int_{Q_{R,\mathcal{T}}(x_0,\mathcal{T})}(u-k)_{\pm}^p|\nabla\phi|^p\,dz\\
        &\quad+c\int_{0}^{\mathcal{T}}\int_{\bbR^n\setminus B_R(x_0)}\int_{B_R(x_0)}\frac{(u-k)_{\pm}^{p-1}(y,t)((u-k)_{\pm}\phi^p)(x,t)}{|x-y|^{n+sp}}\,dx\,dy\,dt
    \end{align*}
    whenever
\begin{equation*}
\begin{aligned}
    \begin{cases}
        k\geq \sup\limits_{\mathcal{Q}}g&\quad\mbox{if $u$ is a sub-solution,}\\
        k\leq \inf\limits_{\mathcal{Q}}g&\quad\mbox{if $u$ is a super-solution.}
    \end{cases}
    \end{aligned}
\end{equation*}
\end{lemma}
\begin{proof}
    Since we choose a cutoff function $\phi$ which depends only on the spatial direction, $\partial_t\phi=0$. In addition, by the choice of $k$, we have 
    \begin{equation*}
        (u-k)_{\pm}(\cdot,0)=0\quad\text{and}\quad \pm\int_{B_R(x_0)}\phi^p\int^u_{k}\beta'_{\epsilon}(\xi)(\xi-k)_{\pm}\,d\xi\,dx\geq 0.
    \end{equation*}
    Therefore, we have the desired estimate.
\end{proof}
We next provide a variant of De Giorgi type argument.
\begin{lemma}\label{lem.di2i}
    Let $u$ be a sub(super)-solution to \eqref{eq.int.appro} and let $\xi\in(0,1/4]$ and $\Theta=(\xi\omega)^{2-p}$. Then there is a constant $\nu_d=\nu_d(n,s,p,\Lambda)$ such that if 
\begin{equation*}
    \pm(\mu_{\pm}-u(\cdot,0))\geq \xi\omega\quad\text{in }B_\rho(x_0),
\end{equation*}
then
\begin{equation*}
    \pm(\mu_{\pm}-u)\geq \xi\omega/2\quad\text{in }B_{\rho/2}(x_0)\times(0,\nu_d\Theta\rho^{sp}],
\end{equation*}
whenever 
\begin{align*}
       (\rho/R)^{\frac{sp}{p-1}} \mathrm{Tail}((u-\mu_{\pm})_{\pm};\mathcal{Q})\leq \xi\omega\quad\text{and}\quad\xi\omega\leq \pm\mu_{\pm}-\sup_{\mathcal{Q}}(\pm g)
    \end{align*}
    and $B_\rho(x_0)\times (0,\nu_d\Theta\rho^{sp}]\subset \mathcal{Q}$.
\end{lemma}
\begin{proof}
    By Lemma \ref{lem.di2} with $t_1$ replaced by 0 and Lemma \ref{lem.energy3}, we deduce the desired result.
\end{proof}

\subsection{Uniform continuity}
In this subsection, we will prove Theorem \ref{thm.intilast}. 
Before that, we first construct a sequence of cylinders $Q_i$ and derive oscillation estimates on such cylinders (see \eqref{seq.cylinderi} and \eqref{iter.ep3}). 
Let us fix $B_{\mathcal{R}}(x_0)\times (0,\mathcal{R}^{sp}]\subset \Omega'_T$ with $x_0\in \overline{\Omega}$. We next set 
\begin{equation*}
    \omega_0\coloneqq\max\{\|u\|_{L^\infty(B_{\mathcal{R}}(x_0)\times (0,\mathcal{R}^{sp}])}+\mathrm{Tail}(u;B_\mathcal{R}(x_0)\times (0,\mathcal{R}^{sp}]),1\}.
\end{equation*}
Let us choose $\rho_0<\mathcal{R}$ such that
\begin{equation*}
    B_{\rho_0}(x_0)\times(0,\rho_0^{sp}(\omega_0/4)^{2-p}]\subset B_\mathcal{R}(x_0)\times (0,\mathcal{R}^{sp}].
\end{equation*}
Take sequences
\begin{equation}\label{seq.rhob}
    \rho_{i+1}\coloneqq \rho_i/N\quad\text{and}\quad \omega_{i+1}\coloneqq\max\left\{m\omega_i,2\osc_{Q_i}g\right\},
\end{equation}
where $N\geq4$, $m=\max\{7/8,2^{-s}\}$ and
\begin{equation}\label{seq.cylinderi}
    {Q_i}\coloneqq B_{\rho_i}(x_0)\times (0,\rho_i^{sp}\theta_i],\quad {\theta_i}\coloneqq(\omega_i/4)^{2-p},\quad\mu_{i}^+\coloneqq \sup_{Q_{i}}u\quad\text{and}\quad \mu_{i}^-\coloneqq \sup_{Q_{i}}u-\omega_i.
\end{equation}
We are now going to determine a constant $N$ to see that
\begin{equation}\label{iter.ep3}
    \osc_{Q_i}u\leq \omega_i\quad\text{and}\quad Q_{i}\subset Q_{i-1} \quad\text{for each }i\geq0,
\end{equation}
where $Q_{-1}=B_{\mathcal{R}}(x_0)\times (0,\mathcal{R}^{sp}]$.
Suppose that 
\begin{equation*}
    \osc_{Q_i}u\leq \omega_i\quad\text{and}\quad Q_{i}\subset Q_{i-1} 
\end{equation*}
hold for any $i=0,1,\ldots j-1$.
As in the proof of \eqref{ineq.tail}, there is a constant $c_2=c_2(n,s,p)\geq 1$ such that
\begin{equation}\label{ineq.tail2i}
    \mathrm{Tail}((u-\mu_{j-1}^{\pm})_\pm;Q_{i-1})\leq c_2\omega_{j-1}.
\end{equation}

For a convenience of notations, we write 
\begin{equation*}
    \rho\coloneqq \sigma R,\quad R\coloneqq \rho_{j-1},\quad \mathcal{Q}\coloneqq Q_{j-1},\quad\omega\coloneqq \omega_{j-1},\quad \mu_+\coloneqq \mu_{j-1}^{+}\quad\text{and}\quad \mu_-\coloneqq \mu_{j-1}^{-},
\end{equation*}
where $\sigma\in(0,1/4]$.
In addition, by Lemma \ref{lem.trans}, we may assume $z_0=0$. 
Let us consider the two cases:
\begin{align}\label{alt.i}
    \mu_{+}-\omega/4\geq \sup_{\mathcal{Q}}g\quad\text{and}\quad \mu_-+\omega/4\leq \inf_{\mathcal{Q}}g.
\end{align}
This is equivalent to
\begin{align*}
    \omega/4\leq \pm\mu_{\pm}-\sup_{\mathcal{Q}}(\pm g).
\end{align*}
We now take 
\begin{align}\label{choi.sigmai}
    \sigma\coloneqq 1/(16c_2)^{\frac{p-1}{sp}}(\leq 1/2),
\end{align}
where the constant $c_2$ is determined in \eqref{ineq.tail2i}.
Since 
\begin{align*}
    \pm(\mu_{\pm}-u(\cdot,0))\geq\pm\mu_{\pm}-\sup_{\mathcal{Q}}(\pm g)\geq \omega/4\quad\text{in }B_\rho
\end{align*}
and
\begin{align*}
    (\rho/R)^{\frac{sp}{p-1}}\mathrm{Tail}((u-\mu_{\pm})_\pm;\mathcal{Q})\leq \omega/4,
\end{align*}
which follows from \eqref{ineq.tail2i} and \eqref{choi.sigmai},
we are now able to use Lemma \ref{lem.di2i} to get that 
\begin{align*}
\pm(\mu_{\pm}-u)\geq \omega/8\quad\text{in } B_{\rho/2}\times(0,\nu_d(\omega/4)^{2-p}\rho^{sp}]
\end{align*}
for some constant $\nu_d=\nu_d(n,s,p,\Lambda)\in(0,1)$.

Therefore, by taking
\begin{align}\label{eq:N}
N\coloneqq\frac{2^{\frac{p-2}{sp}+2}}{\sigma\nu_d^{\frac1{sp}}}(\geq 4),
\end{align}
where the constant $\sigma$ is determined in \eqref{choi.sigmai} and $\nu_d$ is from Lemma \ref{lem.di2i}, we have 
\begin{align}\label{nchoi}
\osc_{Q_{j}}u\leq\osc_{B_{\rho/2}\times (0,\nu_d(\omega/4)^{2-p}\rho^{sp}]}u\leq 7\omega/8\underset{\eqref{seq.rhob}}{\leq} \omega_j.
\end{align}
Indeed, note that $Q_j=B_{\rho_j}\times(0,\rho_j^{sp}\theta_j]$, $\rho_j=\rho_{j-1}/N$ and $\rho=\sigma\rho_{j-1}$. By \eqref{eq:N}, we have 
\begin{align*}
    \rho_{j}=\rho_{j-1}/N\leq \sigma\rho_{j-1}/2\leq \rho/2.
\end{align*}
Next, we note from \eqref{seq.rhob} and $m\geq1/2$ that
\begin{align*}
    \rho_j^{sp}\theta_j=(\rho_{j-1}/N)^{sp}(\omega_{j}/4)^{2-p}\leq (\rho_{j-1}/N)^{sp}(m\omega_{j-1}/4)^{2-p}\leq \nu_d(\omega/4)^{2-p}\rho^{sp}.
\end{align*}
Also, we have 
\begin{align*}
Q_{j}\subset Q_{j-1}.
\end{align*}
Indeed, one can see that $\rho_{j}^{sp}(\omega_j/4)^{2-p}\leq (\rho_{j-1}/N)^{sp} m^{2-p}(\omega_{j-1}/4)^{2-p}\leq \rho_{j-1}^{sp}(\omega_{j-1}/4)^{2-p}$, which gives the above inclusion. If two cases given in \eqref{alt.i} do not occur, then we have 
\begin{align*}
    \osc_{Q_j}u\leq \osc_{Q_{j-1}}u\leq 2\osc_{Q_{j-1}}g\leq \omega_j.
\end{align*}
This proves \eqref{iter.ep3}.

We are ready to prove Theorem \ref{thm.intilast}.
\begin{proof}[Proof of Theorem \ref{thm.intilast}.]
Note that $\omega_0=1$ and that 
\begin{align*}
    \omega_{g}(\rho)\leq c_g(1+\ln(\mathcal{R}/\rho))^{-\delta}
\end{align*}
for some constants $\delta\in(0,1)$ and $c_g\geq1$.
With an inductive process, we are going to derive 
\begin{align}\label{ind.b}
    \omega_i\leq c(1+\ln(\mathcal{R}/\rho_i))^{-\delta},
\end{align}where $\rho_0\coloneqq\mathcal{R}/4^{\frac{p-1}{sp}}$ and $c=c(n,s,p,\Lambda,\delta,c_g)\geq2$. Indeed, by the choice of $\rho_0$, we have $B_{\rho_0}(x_0)\times(0,\rho_0^{sp}(\omega_0/4)^{2-p}]\subset B_{\mathcal{R}}(x_0)\times (0,\mathcal{R}^{sp}]$. We find a positive integer $i_g=i_g(n,s,p,\Lambda,\delta)$ such that for any $i\geq i_g$,
\begin{align}\label{choi.ig}
    \left[1+(i+1)\ln\left(4^{\frac{p-1}{sp}}N\right)\right]/\left[1+i\ln\left(4^{\frac{p-1}{sp}}N\right)\right]\leq 1/m^{\frac1\delta},
\end{align}
as the left-hand side is decreasing with respect to the positive integer $i$ and goes to 1 as $i\to\infty$, where the constants $m$ and $N$ are determined in \eqref{seq.rhob} and \eqref{eq:N}, respectively.
Then there is a constant $c\geq 2c_g$ depending only on $n,s,p,\Lambda$ and $\delta$ such that 
\begin{align*}
    \omega_i\leq c\left(\ln\left(4^{\frac{p-1}{sp}}\rho_0/\rho_{i_g}\right)\right)^{-\delta}= c(1+\ln(\mathcal{R}/\rho_{i_g}))^{-\delta}\leq c(1+\ln(\mathcal{R}/\rho_i))^{-\delta}
\end{align*}
for any $i\leq i_g$.
We now use the inductive argument to prove \eqref{ind.b} for any $i>i_g$. Suppose \eqref{ind.b} holds for any $i=0,1,\ldots,i_g,\ldots j$. If $\omega_{j+1}=2\osc_{Q_j}g$, then the fact that $c\geq 2c_g$ implies
\begin{align*}
    \omega_{j+1}\leq 2c_g\left(1+\ln(\mathcal{R}/\rho_j)\right)^{-\delta}\leq c\left(1+\ln(\mathcal{R}/\rho_{j+1})\right)^{-\delta}.
\end{align*}

On the other hand, if $\omega_{j+1}=m\omega_j$, then we obtain 
\begin{align*}
    \omega_{j+1}\leq cm(1+\ln(\mathcal{R}/\rho_i))^{-\delta}\leq c(1+\ln(\rho_0/\rho_{i+1}))^{-\delta},
\end{align*}
where we have used \eqref{seq.rhob} and \eqref{choi.ig}. Therefore, this implies
\begin{align*}
    \osc_{Q_i}u\leq c(1+\ln(\mathcal{R}/\rho_i))^{-\delta}\quad\text{for any }i\geq0.
\end{align*}
Let us fix $r\in(0,\rho_0]$. Then there is a nonnegative integer $i$ such that
\begin{equation*}
    \rho_{i+1}<r\leq \rho_i.
\end{equation*}
Thus, we get that
\begin{align*}
     \osc_{B_r\times [0,r^{sp}(1/4)^{2-p}]}u=\osc_{B_r\times [0,r^{sp}(\omega_0/4)^{2-p}]}u\leq \osc_{Q_{i}}u\leq c(1+\ln(\mathcal{R}/r))^{-\delta}
\end{align*}
holds for any $r\in(0,\rho_0]$ (Recall $\rho_0=\mathcal{R}/4^{\frac{p-1}{sp}}$), where $c=c(n,s,p,\Lambda,\delta,c_g)$. Therefore, we deduce 
\begin{align*}
     \osc_{B_r\times [0,r^{sp}]}u\leq c(1+\ln(\mathcal{R}/r))^{-\delta}\quad\text{for any }r\in(0,\mathcal{R}],
\end{align*}
which implies \eqref{ineq.intilast}. This completes the proof.
\end{proof}

\section{Existence of a continuous solution}\label{sec:6}
In this section, we prove a global continuity estimates of \eqref{eq.diri.appro} and construct a weak solution $u\in C\left(\overline{\Omega}\times[0,T]\right)$ to \eqref{eq.diri}.
First, we obtain a global uniform continuity estimates of \eqref{eq.diri.appro}.
\begin{lemma}\label{lem.globcon}
    Let $\Omega'\Supset\Omega$ and the complement of $\Omega$ satisfy \eqref{ass.density}.
    We next assume that
    \begin{equation*}
        g\in L^p(0,T;W^{s,p}(\Omega'))\cap C\left({\overline{\Omega'}\times[0,T]}\right)\cap L^\infty(0,T;L^\infty(\bbR^n))
    \end{equation*}and that there is a non-decreasing function $\omega_g:\bbR^+\to\bbR^+$ such that $\omega_g(0)=0$,
\begin{align*}
    \sup_{z_1,z_2\in \overline{\Omega'}\times[0,T]}|g(z_1)-g(z_2)|\leq \omega_g\left(|x_1-x_2|+|t_1-t_2|^{\frac1{sp}}\right).
\end{align*}
Let $u_\epsilon$ be a weak solution to \eqref{eq.diri.appro}. Then there is a constant $\varsigma=\varsigma(n,s,p,\Lambda,\alpha_0)$ such that for any $\delta\in(\varsigma,1)$, if 
\begin{equation}\label{cond.globcon}
    \omega_g(\rho)\leq c_g\left(1+|\ln(1/\rho)|\right)^{-\delta}
\end{equation}
for some constant $c_g\geq1$, then we have 
\begin{align}\label{goal.conti}
    \sup_{z_1,z_2\in\overline{\Omega}\times[0,T]}|u_\epsilon(z_1)-u_\epsilon(z_2)|\leq c\left(\left[1+\left|\ln\left(\frac1{|x_1-x_2|+|t_1-t_2|^{\frac1{sp}}}\right)\right|\right]^{-\varsigma/2}+\epsilon\right),
\end{align}
where $c=c(n,s,p,\Lambda,\alpha_0,\|g\|_{L^\infty},T,\Omega,\Omega',c_g,\delta)$ and $c_1=c_1(n,s,p)$.
\end{lemma}
\begin{proof}
In the proof, the constant $c$ always depends on $n,s,p,\Lambda,\alpha_0,\|g\|_{L^\infty},T,\Omega,\Omega',c_g$ and $\delta$.
Let us recall the constants $\varsigma_1=\varsigma_1(n,s,p,\Lambda)\in (0,1)$ and $\varsigma_2=\varsigma_2(n,s,p,\Lambda,\alpha_0)\in (0,1)$ determined in Theorem \ref{thm.intlast} and \ref{thm.bdylast}, respectively. Let us choose 
\begin{equation*}
    \varsigma\coloneqq \min\{\varsigma_1,\varsigma_2\}.
\end{equation*}
We note that there is a constant $C_0$ such that 
\begin{align*}
    \mathrm{Tail}(g;Q_r(z_0))\leq C_0\|g\|_{L^\infty(\bbR^n\times[0,T])}
\end{align*}
for any $Q_r(z_0)\subset \bbR^n\times[0,T]$.
Let us fix 
\begin{equation}\label{cond.M}
    \mathcal{M}\coloneqq 2C_0\|g\|_{L^\infty(\bbR^n\times[0,T])}.
\end{equation}
From Lemma \ref{lem.trans}, we observe that $\overline{u}_\epsilon\coloneqq u_\epsilon/\mathcal{M}$ is a unique solution to 
\begin{equation*}
\left\{
\begin{alignedat}{3}
\partial_t (\overline{u_\epsilon}+\overline{\beta_\epsilon}(\overline{u_\epsilon}))+\overline{\mathcal{L}}\overline{u_\epsilon}&= 0&&\qquad \mbox{in  $ \Omega_T$}, \\
\overline{u_\epsilon}&=\overline{g}&&\qquad  \mbox{in $\bbR^n\setminus (\partial\Omega\times (0,T]\cup \Omega\times \{t=0\})$},
\end{alignedat} \right.
\end{equation*}
where $\overline{g}=g/\mathcal{M}$, and the function $\overline{\beta_\epsilon}$ and the nonlocal operator $\overline{\mathcal{L}}$ are defined in Lemma \ref{lem.trans}. By the choice of the constants $C_0$ and $\mathcal{M}$, we have
\begin{align}\label{eq:u.epsilon}
    \|\overline{u_\epsilon}\|_{L^\infty(Q_R(z_0))}+\mathrm{Tail}(\overline{u_\epsilon};Q_R(z_0))\leq 1
\end{align}
for any $Q_R(z_0)\subset \bbR^n\times[0,T]
$. In addition, we observe 
\begin{align}\label{cond.rbetap}
\overline{\beta}_\epsilon'(\xi)\geq0\quad\text{and}\quad \int_{\bbR}\overline{\beta}_\epsilon'(\xi)\,d\xi=\mathcal{M}.
\end{align}
We point out that the second condition given in \eqref{cond.betap} was only used to show the following inequality:
\begin{equation*}
    \pm\int_{k}^u\beta_{\epsilon}'(\xi)(\xi-k)_{\pm}\,d\xi\leq (u-k)_{\pm}.
\end{equation*}
For example, see \eqref{ineq02.di1}. Therefore, by taking the conditions given in \eqref{cond.rbetap} instead of \eqref{cond.betap}, we are able to obtain all of the results given in Theorem \ref{thm.intlast}, Theorem \ref{thm.bdylast} and Theorem \ref{thm.intilast}, while all the universal constants $c$ determined in the three theorems also depend on the constant $\mathcal{M}$.

We are now going to derive \eqref{goal.conti}.
To this end, we note that there is a constant $R_0=R_0(\Omega,\Omega')\leq1$ such that 
\begin{align*}
    B_{2R_0}(x_0)\subset \Omega'
\end{align*}
for any $x_0\in\overline{\Omega}$.
Let us fix 
\begin{equation}\label{cond.rho0}
    r_0\coloneqq\min\{1/2^{10}, R_0^{10},1/4^{\frac{20(p-2)}{sp}},T^{\frac{10}{sp}}\}\leq1,
\end{equation}
and write
\begin{equation}\label{r.globcon}
    r\coloneqq |x_1-x_2|+|t_1-t_2|^{\frac1{sp}}.
\end{equation}
We divide the proof of \eqref{goal.conti} into several cases. We may assume $t_1>t_2$ and we write $d(x)\coloneqq \mathrm{dist}(x,\partial\Omega)$.

\begin{enumerate}
    \item $r\geq r_0$: In this case, together with \eqref{eq:u.epsilon}, \eqref{goal.conti} follows by taking constant $c$ sufficiently large.
    \item $r<r_0$ and $\max\left\{d(x_1),d(x_2),t_1^{\frac1{sp}}\right\}<r^{\frac12}$: Then there is the point $y_0\in\partial\Omega$ such that 
    \begin{equation*}
        d(x_1)=|x_1-y_0|.
    \end{equation*}
    Then we observe from \eqref{cond.rho0} that 
    \begin{align*}
        z_1,z_2 \in  B_{r^{1/4}}(y_0)\times (0,r^{{sp}/{4}}]\subset B_{r_0^{1/5}}(y_0)\times (0,r_0^{{sp}/{5}}]\subset \Omega'_T.
    \end{align*}
    By \eqref{cond.globcon} and \eqref{cond.rho0}, we have 
    \begin{equation}\label{ineq1.globcon}
        \omega_g(\rho)\leq c\left(1+\ln(r_0^{1/5}/r)\right)^{-\delta}\quad\text{for any }r\in(0,r_0^{1/5}].
    \end{equation}
    Using \eqref{ineq1.globcon} and Theorem \ref{thm.intilast} with $\mathcal{R}$ and $r$ replaced by $r_0^{\frac15}$ and $r^{\frac14}$, respectively, we have
    \begin{align*}
        |\overline{u_\epsilon}(z_1)-\overline{u_\epsilon}(z_2)|\leq \osc_{B_{r^{\frac14}}(y_0)\times (0,r^{\frac{sp}{4}}]}\overline{u_\epsilon}\leq c\left(1+\ln\left(r_0^{1/5}/r^{1/4}\right)\right)^{-\varsigma/{2}}\leq c(1+\ln(1/r))^{-\varsigma/2}.
    \end{align*}
    \item $r<r_0$ and $\max\left\{d(x_1),d(x_2),t_1^{\frac1{sp}}\right\}\geq r^{\frac12}$:
    \begin{enumerate}
    \item[(c)-(1)]$d(x_i)\geq r^{\frac12}$ for some $i\in\{1,2\}$ and $t_1^{\frac1{sp}}\geq r^{\frac12}$:
    By \eqref{cond.rho0}, we observe 
    \begin{align*}
        z_1,z_2\in Q^{(\omega_0/4)^{2-p}}_{r^{\frac78}}(x_i,t_1)\subset Q^{(\omega_0/4)^{2-p}}_{r^{\frac{13}{16}}}(x_i,t_1)\subset Q_{r^{\frac34}}(x_i,t_1)\subset \Omega_T,
    \end{align*}
     where we write 
    \begin{equation*}
        \omega_0\coloneqq\max\{\|\overline{u_\epsilon}\|_{L^\infty(Q_{r^{3/4}}(x_i,t_1))}+\mathrm{Tail}(\overline{u_\epsilon};Q_{r^{3/4}}(x_i,t_1)),1\}=1
    \end{equation*}
    Then applying Theorem \ref{thm.intlast} with $\mathcal{R}=r^{3/4}$, $\rho_0=r^{13/16}$, we have 
    \begin{align*}
        |\overline{u_\epsilon}(z_1)-\overline{u_\epsilon}(z_2)|\leq \osc_{Q^{(\omega_0/4)^{2-p}}_{r^{\frac78}}(x_i,t_1)}\overline{u_\epsilon}\leq c\left(1+\ln(1/r^{1/16})\right)^{-\varsigma/2}+4\epsilon\leq c(1+\ln(1/r))^{-\varsigma/2}+4\epsilon.
    \end{align*}
        \item[(c)-(2)] $d(x_i)\geq r^{\frac12}$ for some $i\in\{1,2\}$ and $t_1^{\frac1{sp}}<r^{\frac12}$:
        Suppose $t_1^{\frac1{sp}}\leq r^{\frac34}$.
        Then we get 
        \begin{align*}
            z_1,z_2\in B_{r^{\frac34}}(x_i)\times (0,r^{\frac{3sp}4}]\subset B_{r_0^{\frac12}}(x_i)\times (0,r_0^{\frac{sp}2}]\subset\Omega'_T.
        \end{align*}
        Indeed, we get \eqref{ineq1.globcon} with $r_0^{1/5}$ replaced by $r_0^{1/2}$. Therefore, Theorem \ref{thm.intilast} implies
        \begin{align*}
        |\overline{u_\epsilon}(z_1)-\overline{u_\epsilon}(z_2)|\leq \osc_{B_{r^{\frac34}}(x_i)\times (0,r^{\frac{3sp}4}]}\overline{u_\epsilon} \leq c(1+\ln(1/r))^{-\varsigma/2}.
    \end{align*}
     We next suppose $r^{\frac34}<t_1^{\frac1{sp}}<r^{\frac12}$. Then we have 
    \begin{align*}
        z_1,z_2\in Q^{(\omega_0/4)^{2-p}}_{r^{\frac78}}(x_i,t_1)\subset Q^{(\omega_0/4)^{2-p}}_{r^{\frac{13}{16}}}(x_i,t_1)\subset Q_{r^{\frac34}}(x_i,t_1)\subset \Omega_T,
    \end{align*}
    where we write 
    \begin{equation*}
        \omega_0\coloneqq\max\{\|\overline{u_\epsilon}\|_{L^\infty(Q_{r^{3/4}}(x_i,t_1))}+\mathrm{Tail}(\overline{u_\epsilon};Q_{r^{3/4}}(x_i,t_1)),1\}=1.
    \end{equation*}
    Thus by Theorem \ref{thm.intlast}, we obtain
    \begin{align*}
        |\overline{u_\epsilon}(z_1)-\overline{u_\epsilon}(z_2)|\leq \osc_{Q^{(\omega_0/4)^{2-p}}_{r^{\frac78}}(x_i,t_1)}\overline{u_\epsilon}\leq c(1+\ln(1/r))^{-\varsigma/2}+4\epsilon.
    \end{align*}
        \item[(c)-(3)] $d(x_1),d(x_2)<r^{\frac12}$ and $t_1^{\frac1{sp}}\geq r^{\frac12}$.
   Then there is the point $y_0\in\partial\Omega$ such that 
   \begin{equation*}
       d(x_1)=|x_1-y_0|.
   \end{equation*}
   Suppose $t_1^{\frac1{sp}}\geq r^{\frac15}$
   \begin{align*}
       z_1,z_2\in Q_{r^{\frac14}}(y_0,t_1)\subset Q_{r^{\frac15}}(y_0,t_1)\subset \Omega'_T.
   \end{align*}
   By \eqref{cond.globcon}, we have a constant $c$ such that
   \begin{align*}
       \omega_g(\rho)\leq c\left(1+\ln(r^{\frac15}/\rho)\right)^{-\delta}.
   \end{align*}
        Thus, in light of Theorem \ref{thm.bdylast}, we have 
        \begin{align*}
        |\overline{u_\epsilon}(z_1)-\overline{u_\epsilon}(z_2)|\leq \osc_{Q_{r^{\frac14}}(y_0,t_1)}\overline{u_\epsilon}\leq c(1+\ln(1/r))^{-\varsigma/2}+4\epsilon.
    \end{align*}
    On the other hand, if $r^{\frac12}\leq t_1^{\frac1{sp}}< r^{\frac15}$, then we have 
    \begin{align*}
        z_1,z_2\in  B_{r^{\frac16}}(y_0)\times(0,r^{\frac{sp}{6}}]\subset B_{r_0^{1/7}}(y_0)\times(0,r_0^{{sp}/{7}}]\subset \Omega'_T.
    \end{align*}
    Using \eqref{ineq1.globcon} with $r_0^{1/5}$ replaced by $r_0^{1/7}$ and Theorem \ref{thm.intilast}, we have
    \begin{align*}
        |\overline{u_\epsilon}(z_1)-\overline{u_\epsilon}(z_2)|\leq \osc_{B_{r^{1/6}}(y_0)\times(0,r^{{sp}/{6}}]}\overline{u_\epsilon}\leq c(1+\ln(1/r))^{-\varsigma/2}.
    \end{align*}
    \end{enumerate}
\end{enumerate}
Combining all the cases together with \eqref{cond.M}, \eqref{r.globcon} and the fact that $\overline{u_\epsilon}=u/\mathcal{M}$ yields \eqref{goal.conti}.
\end{proof}

We are now ready to prove Theorem \ref{thm.exist}. We point out that our argument is based on the proof of \cite[Theorem 1.1]{BarKuuLinUrb18}
\begin{proof}[Proof of Theorem \ref{thm.exist}.]
    In light of the Ascoli-Arzel\'a-type argument given in \cite[Section 5]{BarKuuLinUrb18} together with \eqref{goal.conti}, there exists a sequence $\{u_i\}$ which is a solution to \eqref{eq.diri.appro} with $\epsilon$ replaced by $1/i$ such that $u_i$ uniformly converges to $u$ in $\overline{\Omega}\times [0,T]$ and $u\in C\left(\overline{\Omega}\times [0,T]\right)$. We are now going to prove that $u$ is a local weak solution to \eqref{eq.main}. 
    We first show  
    \begin{equation}\label{goal1.exist}
        u\in L^p(0,T;W^{s,p}(\Omega'))\cap L^\infty(0,T;L^2(\Omega))\cap L^\infty(0,T;L^{p-1}_{sp}(\bbR^n))\cap L^{\infty}_{\loc}(0,T;L^{\infty}_{\loc}(\Omega)).
    \end{equation}
    By testing $u_i-g$ to \eqref{eq.diri.appro} with $\epsilon=1/i$, we have for any $\tau\in(0,T]$
    \begin{equation*}
    \begin{aligned}
        J\coloneqq&\int_{0}^{\tau}\int_{\Omega}\partial_t(u_i-g)(u_i-g)\,dz+\int_{0}^{\tau}\int_{\Omega}\partial_t\beta_{1/i}(u_i)(u_i-g)\,dz\\
        &\quad+\int_{0}^{\tau}\int_{\bbR^{n}}\int_{\bbR^n}\left[\phi(u_i(x,t)-u_i(y,t))-\phi(g(x,t)-g(y,t))\right]\\
        &\qquad \times \frac{\left[u_i(x,t)-u_i(y,t)-(g(x,t)-g(y,t))\right]}{|x-y|^{n+sp}}k(x,y,t)\,dx\,dy\,dt\\
        &=\int_{0}^{\tau}\int_{\bbR^{n}}\int_{\bbR^n}\frac{[\phi(g(x,t)-g(y,t))]\left[u_i(x,t)-u_i(y,t)-(g(x,t)-g(y,t))\right]}{|x-y|^{n+sp}}k(x,y,t)\,dx\,dy\,dt\\
        &\eqqcolon J_1,
    \end{aligned}
    \end{equation*}
   where we write $\phi(t)=|t|^{p-2}t$. We obtain 
   \begin{align*}
       J\geq \int_{\Omega}\frac{|(u_i-g)(\tau,x)|^2}{2}\,dx+{\underbrace{\int_{0}^\tau\int_{\Omega}\partial_t\beta_{1/i}(u_i)(u_i-g)\,dz}_{\coloneqq \widetilde{J}}}+\frac{1}{\Lambda}[u_i-g]^p_{L^{p}(0,\tau;W^{s,p}(\bbR^n))}
   \end{align*}
   for any $\tau\in(0,T]$.
   In addition, by integration by parts, we further estimate $\widetilde{J}$ as
   \begin{align*}
       \widetilde{J}&=\left(\int_{\Omega}\int^{u_i(x,\tau)}_0\beta_{1/i}'(\xi)\xi\,d\xi\,dx-\int_{\Omega}\int^{u_i(x,0)}_0\beta_{1/i}'(\xi)\xi\,d\xi\,dx\right)\\
       &+\int_{0}^\tau\int_{\Omega}\beta_{1/i}(u_i)\partial_tg \,dz-\int_{\Omega}(\beta_{1/i}(u_i)g)(x,\tau)\,dx+\int_{\Omega}(\beta_{1/i}(u_i)g)(x,0)\,dx,
   \end{align*}
   where we have used the fact that 
   \begin{align*}
       \int_{0}^{\tau}\int_{\Omega}\partial_t\beta_{1/i}(u_i)u_i\,dz=\int_{0}^{\tau}\partial_t\left(\int_{\Omega}\int_{0}^{u_\epsilon(x,t)}\beta'_{1/i}(\xi)\xi\,d\xi\,dx\right)\,dt.
   \end{align*}
   Therefore, we have 
   \begin{align*}
       \widetilde{J}\geq -\left(\int_{0}^T\int_{\Omega}|\partial_tg|^2\,dz+\sup_{t\in[0,T]}\int_{\Omega}|g(x,t)|^2\,dx+c\right)
   \end{align*}
   for some constant $c=c(n,T,\Omega)$.

On the other hand, we estimate $J_1$ as 
   \begin{align*}
       J_1&\leq c[g]^p_{L^p(0,\tau;W^{s,p}(\Omega'))}+\frac{1}{4\Lambda}[u_i-g]^p_{L^p(0,\tau;W^{s,p}(\Omega'))}+c\|g\|_{L^p(\Omega_\tau)}\|u_i-g\|^{p-1}_{L^p(\Omega_\tau)}\\
       &\quad+c\|g\|^{p-1}_{L^{p-1}(0,\tau;L^{p-1}_{sp}(\bbR^n))}\sup_{t\in(0,\tau)}\int_{\Omega}|(u_i-g)(x,t)|\,dx.
   \end{align*}
   By the fact that 
   \begin{equation}\label{ineq1.exist}
       \|u_i-g\|_{L^p(\Omega_\tau)}\leq c[u_i-g]_{L^p(0,\tau;W^{s,p}(\Omega'))}
   \end{equation}
for some constant $c=c(n,s,p,\Omega,\Omega',\alpha_0)$ which follows from \cite[Lemma 4.7]{Coz17} together with \eqref{ass.density}, and by using Young's inequality, we further estimate $J_1$ as 
\begin{align*}
    J_1&\leq c\|g\|^p_{L^p(0,\tau;W^{s,p}(\Omega'))}+\frac{1}{2\Lambda}[u_i-g]^p_{L^p(0,\tau;W^{s,p}(\Omega'))}+\frac{1}{4}\sup_{t\in[0,\tau]}\int_{\Omega}|(u_i-g)(x,t)|^2\,dx\\
    &\quad+c\|g\|^{2(p-1)}_{L^{p-1}(0,\tau;L^{p-1}_{sp}(\bbR^n))}
\end{align*}
for some constant $c=c(n,s,p,\Omega,\Omega',\alpha_0)$. We now combine all of the estimate $J$, $\widetilde{J}$ and $J_1$ to see that 
\begin{align*}
    \sup_{\tau\in[0,T]}\int_{\Omega}{|(u_i-g)(\tau,x)|^2}\,dx+[u_i-g]^p_{L^{p}(0,T;W^{s,p}(\bbR^n))}\leq c,
\end{align*}
where we have used \eqref{cond.g} with $\partial_tg\in L^2(\Omega_T)$. This implies 
\begin{align*}
    \sup_{\tau\in[0,T]}\int_{\Omega}{|u_i(\tau,x)|^2}\,dx+[u_i]^p_{L^{p}(0,T;W^{s,p}(\Omega'))}\leq c.
\end{align*}
Moreover, after a few simple computations together with \eqref{ineq1.exist}, we get 
\begin{equation}\label{ineq2.exist}
    \sup_{\tau\in[0,T]}\int_{\Omega}{|u_i(\tau,x)|^2}\,dx+\|u_i\|^p_{L^{p}(0,T;W^{s,p}(\Omega'))}\leq c,
\end{equation}
where the constant $c$ is independent of $i$.
Applying Fatou's lemma into \eqref{ineq2.exist} along the fact that $u_i\to u$ in $C\left(\overline{\Omega}\times [0,T]\right)$ and $u\equiv g$ on $(\bbR^n\setminus \Omega)\times (0,T)$, we get \eqref{goal1.exist}.

We next prove that there is a function 
\begin{equation*}
    v\in u+\beta(u)
\end{equation*}
such that for any $t_1,t_2\in(0,T]$ and $\oldphi\in L^p(t_1,t_2;W_0^{s,p}(\Omega,\Omega'))\cap W^{1,2}(t_1,t_2;L^2(\Omega))$,
\begin{align*}
    &\int_{\Omega}(v\oldphi)(x,\tau)\,dx\Bigg\rvert_{t_1}^{t_2}-\int_{t_1}^{t_2}\int_{\Omega}v\partial_t\oldphi\,dz\\
    &\quad+\int_{t_1}^{t_2}\int_{\bbR^n}\int_{\bbR^n}\frac{\phi(u(x,t)-u(y,t))(\oldphi(x,t)-\oldphi(y,t))}{|x-y|^{n+sp}}k(x,y,t)\,dx\,dy\,dt=0.
\end{align*}
Since 
\begin{equation*}
    \left\|\frac{\phi(u_i(x,t)-u_i(y,t))}{|x-y|^{\frac{n+sp}{p'}}}\right\|_{L^{p'}(\Omega'\times \Omega'\times (0,T))}<\infty
\end{equation*}
and 
\begin{equation*}
    \lim_{i\to\infty}\frac{\phi(u_i(x,t)-u_i(y,t))}{|x-y|^{\frac{n+sp}{p'}}}\to \frac{\phi(u(x,t)-u(y,t))}{|x-y|^{\frac{n+sp}{p'}}}\quad\text{for a.e.}\,\,(x,y,t)\in\Omega'\times\Omega'\times (0,T),
\end{equation*}
by the weak compactness theorem, there is a sequence $\{u_i\}$ (up to a labeling) such that
\begin{align}\label{limit1.exist}
    \frac{\phi(u_i(x,t)-u_i(y,t))}{|x-y|^{\frac{n+sp}{p'}}}k(x,y,t)^{}\rightharpoonup \frac{\phi(u(x,t)-u(y,t))}{|x-y|^{\frac{n+sp}{p'}}}k(x,y,t)\quad\text{in }L^{p'}(\Omega'\times \Omega'\times (0,T)).
\end{align}
Next, we observe for any $\phi\in L^p(t_1,t_2;W_0^{s,p}(\Omega,\Omega'))$, 
\begin{align*}
    &\lim_{i\to\infty}\int_{t_1}^{t_2}\int_{\bbR^n}\int_{\bbR^n}\frac{\phi(u_i(x,t)-u_i(y,t))(\oldphi(x,t)-\oldphi(y,t))}{|x-y|^{n+sp}}k(x,y,t)\,dx\,dy\,dt\\
    &=\lim_{i\to\infty}\int_{t_1}^{t_2}\int_{\Omega'}\int_{\Omega'}\frac{\phi(u_i(x,t)-u_i(y,t))(\oldphi(x,t)-\oldphi(y,t))}{|x-y|^{n+sp}}k(x,y,t)\,dx\,dy\,dt\\
    &\quad+2\lim_{i\to\infty}\int_{t_1}^{t_2}\int_{\bbR^n\setminus \Omega'}\int_{\Omega}\frac{\phi(u_i(x,t)-u_i(y,t))(\oldphi(x,t)-\oldphi(y,t))}{|x-y|^{n+sp}}k(x,y,t)\,dx\,dy\,dt\eqqcolon J_1+J_2.
\end{align*}
By \eqref{limit1.exist} together with the fact that 
\begin{align*}
    \frac{\oldphi(x,t)-\oldphi(y,t)}{|x-y|^{\frac{n+sp}{p}}}\in L^p(\Omega'\times \Omega'\times (0,T)),
\end{align*} we have 
\begin{align*}
    J_1=\int_{t_1}^{t_2}\int_{\Omega'}\int_{\Omega'}\frac{\phi(u(x,t)-u(y,t))(\oldphi(x,t)-\oldphi(y,t))}{|x-y|^{n+sp}}k(x,y,t)\,dx\,dy\,dt.
\end{align*}

On the other hand, using the fact that $x\in \Omega$ and $y\in \bbR^n\setminus \Omega'$, $u_i\to u$ in $C(\overline{\Omega}\times [0,T])$ and $u_i(y)=g(y)=u(y)$, we get 
\begin{align*}
    J_2=2\int_{t_1}^{t_2}\int_{\bbR^n\setminus \Omega'}\int_{\Omega}\frac{\phi(u(x,t)-u(y,t))(\oldphi(x,t)-\oldphi(y,t))}{|x-y|^{n+sp}}k(x,y,t)\,dx\,dy\,dt.
\end{align*}
Therefore, we have 
\begin{equation}\label{limit10.exist}
    \begin{aligned}
    &\lim_{i\to\infty}\int_{t_1}^{t_2}\int_{\bbR^n}\int_{\bbR^n}\frac{\phi(u_i(x,t)-u_i(y,t))(\oldphi(x,t)-\oldphi(y,t))}{|x-y|^{n+sp}}k(x,y,t)\,dx\,dy\,dt\\
    &=\int_{t_1}^{t_2}\int_{\bbR^n}\int_{\bbR^n}\frac{\phi(u(x,t)-u(y,t))(\oldphi(x,t)-\oldphi(y,t))}{|x-y|^{n+sp}}k(x,y,t)\,dx\,dy\,dt.
    \end{aligned}
\end{equation}

We are now going to find a function $v$ such that
\begin{equation*}
        \{v(z)\,:\,z\in \Omega_T\}\subset \{(u+\beta(u))(z)\,:\,z\in\Omega_T\}
    \end{equation*}
and
\begin{equation}\label{limit11.exist}
\begin{aligned}
    &\lim_{i\to\infty}\left[\int_{\Omega}\left(\left(u_i+\beta_{1/i}(u_i)\right)\oldphi\right)(x,\tau)\,dx\Bigg\rvert_{t_1}^{t_2}+\int_{t_1}^{t_2}\int_{\Omega}\left(u_i+\beta_{1/i}(u_i)\right)\partial_t\oldphi\,dz\right]\\
    &=\int_{\Omega}\left(v\oldphi\right)(x,\tau)\,dx\Bigg\rvert_{t_1}^{t_2}+\int_{t_1}^{t_2}\int_{\Omega}v\partial_t\oldphi\,dz.
\end{aligned}
\end{equation}
Let us define
\begin{equation*}
    \mathcal{A}_1\coloneqq \{z\in \Omega_T\,:\, u(z)>0\}\quad\text{and}\quad \mathcal{A}_2\coloneqq \{z\in \Omega_T\,:\, u(z)<0\}.
\end{equation*}
First, we prove
\begin{align}\label{limit2.exist}
    \lim_{i\to\infty}\beta_{1/i}(u_i)(z)=1\quad\text{for any }z\in \mathcal{A}_1
\end{align}
and
\begin{align}\label{limit3.exist}
    \lim_{i\to\infty}\beta_{1/i}(u_i)(z)=0\quad\text{for any }z\in \mathcal{A}_2.
\end{align}
Let us assume $z\in \mathcal{A}_1$. Then $u(z)>\delta$ for some $\delta>0$ and there is a positive integer $l$ such that if $i\geq l$, then $u_i(z)>\frac{\delta}{2}$, as $u_i$ uniformly converges to $u$. Thus, we get
\begin{equation*}
    \beta_{1/i}(u_i(z))=1\quad\text{if } i>\max\{2/\delta,l\},
\end{equation*}
which implies \eqref{limit2.exist}. We next assume $z\in \mathcal{A}_2$. Then $u(z)<-\delta$ for some $\delta>0$ so there is a positive integer $l$ such that if $i\geq l$, then $u_i(z)<-\frac{\delta}{2}$. Thus, we get
\begin{equation*}
    \beta_{1/i}(u_i(z))=0\quad\text{if } i>\max\{2/\delta,l\},
\end{equation*}
which gives \eqref{limit3.exist}. Let us write $w_i(z)\coloneqq\beta_{1/i}(u_i)(z)\in [0,1]$. Since $\sup\limits_{t\in[0,T]}\sup\limits_{i}\|w_i(\cdot,t)\|_{L^\infty(\Omega)}\leq 1$, by the weak compactness theorem, there is a function $w(x,t)\in[0,1]$ such that 
\begin{equation}\label{wlimit.exist}
    w_i(\cdot,t)\overset{\ast}{\rightharpoonup} w(\cdot,t)\quad\text{in }L^1(\Omega)\quad\text{for every }t\in[0,T].
\end{equation}
In addition, we get 
\begin{equation}\label{limit31.exist}
    w_i\overset{\ast}{\rightharpoonup} w\quad\text{in }L^1(\Omega_T).
\end{equation}
To do this, let us choose $\oldphi \in L^1(\Omega_T)$ and write
\begin{align*}
    A_i(t)\coloneqq \int_{\Omega}(w_i\oldphi)(x,t)\,dx\quad\text{and}\quad g(t)\coloneqq\int_{\Omega}|\oldphi|(x,t)\,dx.
\end{align*}
By the fact that $\oldphi\in L^1((0,T);L^1(\Omega))$, we have $\oldphi(\cdot,t)\in L^1(\Omega)$ a.e. in $(0,T)$. Therefore, we get 
\begin{align}\label{limit4.exist}
    \lim_{i\to\infty}A_i(t)=\lim_{i\to\infty}\int_{\Omega}(w_i\oldphi)(x,t)\,dx=\int_{\Omega}(w\oldphi)(x,t)\,dx\quad\text{a.e. in }(0,T),
\end{align}
where we have used \eqref{wlimit.exist}.

Since $|A_i(t)|\leq g(t)$ and $g(t)\in L^1(0,T)$, by Lebesgue dominated convergence theorem together with \eqref{limit4.exist}, we have 
\begin{align*}
    \lim_{i\to\infty}\int_{\Omega_T}w_i\oldphi\,dz=\lim_{i\to\infty}\int_{0}^{T}A_i(t)\,dt=\int_{0}^{T}\lim_{i\to\infty}A_i(t)\,dt=\int_{0}^T\int_{\Omega}w\oldphi\,dz.
\end{align*}
Therefore, we have verified \eqref{limit31.exist}.
In addition, by \eqref{limit2.exist}, we have $w(z)=1$ a.e., if $z\in \mathcal{A}_1$. Otherwise, there is a set $\mathcal{A}\subset\mathcal{A}_1$ with $|\mathcal{A}|\neq 0$ and $w(z)< 1$ in $z\in \mathcal{A}$. Then we get
\begin{align*}
   |\mathcal{A}|= \lim_{i\to\infty}\int_{\mathcal{A}}\beta_{1/i}(u_i)(z)\,dz=\lim_{i\to\infty}\int_{\mathcal{A}}w_i(z)\,dz=\int_{\mathcal{A}}w(z)\,dz<|\mathcal{A}|,
\end{align*}
which is a contradiction. Similarly, we observe from \eqref{limit3.exist} that $w(z)=0$ a.e., if $z\in \mathcal{A}_2$. By taking $v=u+w$, we prove \eqref{limit11.exist}. Therefore, in light of \eqref{limit10.exist} and \eqref{limit11.exist}, we construct a pair $(u,v)$ which is a weak solution to \eqref{eq.diri}. In addition, by Lemma \ref{lem.globcon} and using the fact that $u_i\to u$ in $C\left(\overline{\Omega}\times[0,T]\right)$, we derive \eqref{ineq.thm.exist}. This completes the proof.
\end{proof}

\printbibliography

\end{document}